\documentclass{amsart}

\usepackage{anysize}
\usepackage{amsthm}
\usepackage{amssymb}
\usepackage{color}
\usepackage{amsmath}
\usepackage{amsfonts}
\usepackage{bbm}
\usepackage{graphicx}
\usepackage{setspace}
\usepackage{textcomp}
\usepackage{wasysym}
\usepackage{subfig}
\usepackage{setspace}
\usepackage{multirow} 
\usepackage{mathrsfs}
\usepackage[all]{xy}
\usepackage{tikz}

\marginsize{2.56cm}{2.56cm}{2cm}{2cm}

\title{Hyperfunctions, the Duistermaat-Heckman theorem, and Loop Groups}
\author{Lisa C. Jeffrey}
\author{James A. Mracek}
\address{Department of Mathematics \\ University of Toronto \\ Bahen Center \\ 40 St. George Street, Room 6290 \\ Toronto, ON \\ Canada \\ M5S2E4}
\email{jmracek@math.toronto.edu \\ jeffrey@math.toronto.edu}

\date{\today}

\newtheorem{thm}{Theorem}
\newtheorem*{define}{Definition}
\newtheorem{lemma}{Lemma}
\numberwithin{lemma}{section}

\numberwithin{ex}{section}
\newtheorem{prop}{Proposition}

\numberwithin{example}{section}

\newtheorem{corollary}{Corollary}

\begin{document}
\pagenumbering{Roman}


\begin{abstract}
In this article we investigate the Duistermaat-Heckman theorem using the theory of hyperfunctions.  In applications involving Hamiltonian torus actions on infinite dimensional manifolds, this more general theory seems to be necessary in order to accomodate the existence of the infinite order differential operators which arise from the isotropy representations on the tangent spaces to fixed points.  We will quickly review of the theory of hyperfunctions and their Fourier transforms.  We will then apply this theory to construct a hyperfunction analogue of the Duistermaat-Heckman distribution.  Our main goal will be to study the Duistermaat-Heckman hyperfunction of $\Omega SU(2)$, but in getting to this goal we will also characterize the singular locus of the moment map for the Hamiltonian action of $T\times S^1$ on $\Omega G$.  The main goal of this paper is to present a Duistermaat-Heckman hyperfunction arising from a Hamiltonian action on an infinite dimensional manifold.
\end{abstract}
\maketitle
\pagebreak
\onehalfspacing
\pagenumbering{arabic}

\section{Introduction}

For finite-dimensional compact symplectic manifolds equipped with a Hamiltonian torus action with moment map $\mu$, the Duistermaat-Heckman theorem gives an explicit formula for an oscillatory integral over the manifold in terms of information about the fixed point set of the torus action, and the action of the torus on the normal bundle to the fixed point set.  Furthermore, the Fourier transform of this integral controls the structure of the cohomology rings of the various symplectic reductions.  For Hamiltonian actions on infinite dimesional symplectic manifolds, little is known is known about the behaviour of their corresponding Duistermaat-Heckman distributions.  In this paper we define the same oscillatory integral for the natural Hamiltonian torus action on the based loop group, as introduced by Atiyah and Pressley, in order to give an expression for a Duistermaat-Heckman \textit{hyperfunction}.  The essential reason for introducing hyperfunction theory is that the local contribution to the Duistermaat-Heckman polynomial near the image of a fixed point is a Green's function for an infinite order differential equation.  Since infinite order differential operators do not act on Schwarz distributions, we are forced to use this more general theory.  After this work had been completed, we learned of the related work of Roger Picken \cite{picken1989propagator}.

The layout of this article is as follows. In Section 2 we review the theory of hyperfunctions, following \cite{kato1999fundamentals, kaneko1989introduction}. In Section 3 we study hyperfunctions that arise naturally from Hamiltonian group actions via localization. Section 4 reviews the based loop group and its Hamiltonian action (introduced by Atiyah and Pressley \cite{atiyahpressley}).  Section 5 describes the fixed point set of any one parameter subgroup of this torus.  In Section 6 we demonstrate the theorems of Section 5 for the based loop group of $SU(2)$.  In Section 7, we compute the isotropy representations of the torus that acts on the based loop group on the tangent spaces to each of the fixed points. Finally, Section 8 applies the hyperfunction localization theorem to $\Omega SU(2)$.

\section{Introduction to Hyperfunctions}

In this section we will quickly review the elements of hyperfunction theory which are needed in order to make sense of the fixed point localization formula for a Hamiltonian action on an infinite dimensional manifold.  We will assume that the reader is familiar with Hamiltonian group actions, but not necessarily with hyperfunctions.  Our exposition will follow a number of sources.  The bulk of the background material follows \cite{kato1999fundamentals, kaneko1989introduction}, while the material on the Fourier transform of hyperfunctions is covered in \cite{kaneko1989introduction} as well as the original paper of Kawai \cite{kawai1970theory}.  The original papers of Sato also give great insight into the motivation for introducing this theory \cite{sato1959theory}.  The lecture notes of Kashiwara, Kawai, and Sato also give useful insight into why hyperfunction and microfunction theory is needed to solve problems in linear partial differential equations \cite{sato1973microfunctions}.

We will let $\mathcal{O}$ be the sheaf of holomorphic functions on $\mathbb{C}^n$.  Points in $\mathbb{C}^n$ will be denoted $z = (z_1,\dots,z_n) = (x_1+iy_1,\dots,x_n+iy_n)$, and we will write $\text{Re}(z) = (x_1,\dots,x_n) \in \mathbb{R}^{n}$ and $\text{Im}(z) = (y_1,\dots,y_n) \in i\mathbb{R}^{n}$.  An \text{open convex cone }$\gamma \subseteq \mathbb{R}^{n}$ is a convex open set such that, for every $c \in \mathbb{R}_{>0}$, if $x \in \gamma$ then $cx \in \gamma$.  We allow $\mathbb{R}^{n}$ itself to be an open convex cone.  If $\gamma$ is an open convex cone, we will denote its polar dual cone by $\gamma^\circ$.  Let $\Gamma$ denote the set of all open convex cones in $\mathbb{R}^n$.  If $\gamma \in \Gamma$ and $\Omega \subseteq\mathbb{R}^n$ is an open set, then we denote $\Omega\times i\gamma = \left\{ (x_j+iy_j) \in \mathbb{C}^{n}\,\vert\, \text{Re}(x) \in \Omega, \text{Im}(z) \in \gamma\right\}$.  An \textit{infinitesimal wedge}, denoted $\Omega\times i\gamma 0$, is a choice of an open subset $U\subseteq \Omega\times i \gamma$ which is asymptotic to the cone opening (we will not need the precise definition, so we omit it).  We will denote the collection of germs of holomorphic functions on the wedge $\Omega\times i \gamma$ by $\mathcal{O}(\Omega \times i\gamma 0)$; that is, we take a direct limit of the holomorphic functions varying over the collection of all infinitesimal wedges $U = \Omega\times i\gamma 0 \subseteq \Omega\times i\gamma$:
\[ \mathcal{O}(\Omega \times i\gamma 0) = \varinjlim_{U\subseteq \Omega\times i \gamma} \mathcal{O}(U) \]
We will use the notation $F(z + i \gamma 0)$ to denote an element of $\mathcal{O}(\Omega \times i\gamma 0)$.

\begin{define}
A hyperfunction on $\Omega\subseteq \mathbb{R}^{n}$ is an element:
\[ \sum_{i = 1}^{n} F(z+i\gamma_i 0) \in \bigoplus_{\gamma \in \Gamma} \mathcal{O}(\Omega \times i\gamma 0) \big/ \sim \]
where the equivalence relation is given as follows.  If $\gamma_1, \gamma_2, \gamma_3 \in \Gamma$ are such that $\gamma_3 \subseteq \gamma_1\cap \gamma_2$ and $F_i \in \mathcal{O}(\Omega\times i\gamma_i)$, then $F_1(z) + F_2(z) \sim F_3(z)$ if and only if $(F_1(z) + F_2(z))\vert_{\gamma_3} = F_3(z)$.  If $\Omega \subseteq \mathbb{R}^{n}$, we will denote the collection of hyperfunctions on $\Omega$ by $\mathcal{B}(\Omega)$.
\end{define}

When we wish to keep track of the cones we will use the notation $f(x) = \sum_j F(z+i\gamma_j 0)$; we call such a sum a boundary value representation of $f(x)$. Alternatively, we will sometimes also use the notation $F(z+i\gamma 0) = b_{\gamma}(F(z))$ when the expression for $F(z)$ makes it notationally burdensome to include the text $+i\gamma0$.  The association $\Omega \mapsto \mathcal{B}(\Omega)$ forms a flabby sheaf on $\mathbb{R}^{n}$, although we will not make use of the sheaf theoretical nature of hyperfunctions in this article.  Actually, what is more, is that this is a sheaf of $D$-modules on $\mathbb{C}^n$; the sheaf of differential operators acts termwise on each element of a sum $\sum_\gamma F(z+i\gamma 0)$.  

The relation defining the sheaf of hyperfunctions allows us to assume that the cones appearing in the sum are disjoint.  Indeed, if we have a hyperfunction $f(x) = F_1(z+i\gamma_1 0) + F_2(z+i\gamma_2 0)$ such that $\gamma_1 \cap \gamma_2 \neq \emptyset$, then we simply observe that we have an equality of equivalence classes:
\[ F_1(z+i\gamma_1 0) + F_2(z+i\gamma_2 0) = (F_1 + F_2)(z + i\gamma_1\cap \gamma_2 0) \]
Similarly, if $\gamma_1\subseteq \gamma_2$ and $F(z)$ is an analytic function on the wedge $\Omega \times i \gamma_1 0$ that admits an analytic extension to $\Omega \times i\gamma_2 0$, then $F(z + i\gamma_1 0) = F(z + i\gamma_2 0)$ as hyperfunctions.  A particular example of this says that two hyperfunctions $f(x) = F_+(z + i0) + F_-(z - i 0) ,g(x) = G_+(z + i0) + G_-(z - i 0) \in \mathcal{B}(\mathbb{R})$ are equal when the function:
\[ F(z) = \left\{\begin{array}{cc} F_+(z) - G_+(z) & \text{Im}(z) > 0 \\ F_-(z) - G_-(z) & \text{Im}(z) < 0 \end{array}\right. \]
admits an analytic extension across the real axis.

The following definition is necessary to define the product of hyperfunctions.  We say that a hyperfunction $f(x)$ is \textit{microanalytic at $(x,\xi) \in T^*\mathbb{R}^n$} if and only if there exists a boundary value representation:
\[ f(x) = \sum_{j=1}^{n} F(z+i\gamma_j 0) \]
such that $\gamma_j \cap \left\{y \in \mathbb{R}^{n}\,\vert\, \xi(y) < 0 \right\} \neq \emptyset$ for every $j \in 1,\dots,n$.  The \textit{singular support} of a hyperfunction $f(x)$, denote $SS(f) \subseteq T^*\mathbb{R}^{n}$, is defined to be the set of points $(x,\xi) \in T^* \mathbb{R}^{n}$ such that $f(x)$ is not microanalytic at $(x,\xi)$.  If $S\subseteq T^*\mathbb{R}^{n}$ then we denote $S^\circ = \left\{ (x,\xi) \in T^*\mathbb{R}^{n}\,:\, (x,-\xi) \in S\right\}$.    

\begin{define}
Suppose that $f,g \in \mathcal{B}(\Omega)$ are two hyperfunctions such that $SS(f)\cap SS(g)^{\circ} = \emptyset$, then the product $f(x)\cdot g(x)$ is the hyperfunction defined by:
\[ f(x)\cdot g(x) = \sum_{j,k} (F_j \cdot G_k) (x+i(\gamma_j\cap \Delta_k)0) \]
where we have chosen appropriate boundary value representations:
\[ f(x) = \sum_{j} F_j(z+i\gamma_j 0) \]
\[ g(x) = \sum_{k} G_j(z+i\Delta_k 0) \]
such that $\gamma_j \cap \Delta_k \neq \emptyset$ for all $j,k$.
\end{define}

In the above definition, the condition on singular support is simply ensuring existence of boundary value representations of $f$ and $g$ such that for all pairs $j,k$ the intersection $\gamma_j \cap \Delta_k \neq \emptyset$ \cite[Theorem~3.2.5]{kaneko1989introduction}.

We may define an infinite product of hyperfunctions when the singular support condition holds pairwise, and the corresponding infinite product of holomorphic functions converges to a holomorphic function.  This result will be necessary to define the equivariant Euler class of the normal bundle to a fixed point in $\Omega G$ as a hyperfunction.  

\begin{lemma}
\label{lemma:infiniteproduct}
If $\left\{F_k(z+i\gamma_k 0)\right\}_{k=1}^{\infty}$ is a sequence of hyperfunctions on $\Omega$ such that:
\begin{enumerate}
	\item For all pairs $j\neq k$, $SS(F_k(z+i\gamma_k 0))\cap SS(F_j(z+i\gamma_k 0))^\circ = \emptyset$  
	\item $\gamma = \displaystyle{\bigcap_{k=1}^{\infty} \gamma_k}$ is open
	\item The infinite product $\displaystyle{\prod_{k=1}^{\infty} F_k(z)}$ is uniformly convergent on compact subsets of $\Omega\times i\gamma$
\end{enumerate}
then there exists a hyperfunction $F(z +i\gamma 0)$ such that:
\[ F(z+i\gamma 0) = \prod_{k=1}^{\infty} F(z + i\gamma_k 0) \]
\end{lemma}

\begin{proof}
The condition on singular support is necessary to define any product of the $F_k$.  Since the intersection of the cones is open, the wedge $\Omega\times i\gamma$ is a well defined open set in $\mathbb{C}^{n}$, and the convergence condition on the infinite product ensures that the following limit is a holomorphic function on $\Omega\times i \gamma$:
\[ F(z) = \lim_{N\to \infty} \prod_{k=1}^{N} F_k(z) \]
This result has shown that the infinite product of the hyperfunctions $F_k(z+i\gamma_k 0)$ is well defined and equal to $F(z+i\gamma 0)$.
\end{proof}

We now describe how to define the Fourier transform of a hyperfunction.  The following two definitions are central to the theory of hyperfunction Fourier transforms.  We will restrict our attention to the class of \textit{Fourier hyperfunctions}, also known as \textit{slowly increasing hyperfunctions}.

\begin{define}
A holomorphic function $F \in \mathcal{O}(\mathbb{R}^n \times i \gamma 0)$ is called slowly increasing if and only if for every compact subset $K \subseteq i \gamma 0$, and for every $\epsilon > 0$, there exist constants $M,C > 0$ such that, for all $z \in \mathbb{R}^{n}\times i K$, if $\vert \text{Re}(z)\vert > M$ then $\vert F_j(z) \vert \leq C \exp(\epsilon \text{Re}(z))$.

A holomorphic function $F \in \mathcal{O}(\mathbb{R}^{n}\times i\gamma 0)$ is called exponentially decreasing on the (not necessarily convex) cone $\Delta \subseteq \mathbb{R}^n$ if and only if there exists $\delta > 0$, such that for every compact $K\subseteq i\gamma_j 0$, and for every $\epsilon > 0$, there exist constants $M, C > 0$ such that for every $z \in \Delta\times i K$, if $\vert \text{Re}(z) \vert > M$ then $\vert F_j(z) \vert \leq C \exp(-(\delta-\epsilon) \text{Re}(z))$.
\end{define}

\noindent\textit{Remarks on the definition}: 
\begin{enumerate}
\item A hyperfunction will be called slowly increasing (resp. exponentially decreasing on $\Delta$) if and only if it admits a boundary value representation:
\[ f(x) = \sum_{j=1}^{n} F_j(z+i\gamma 0) \]
such that each of the $F_j(z)$ is slowly increasing (resp. exponentially decreasing on $\Delta$).
\item If $F(z)$ is slowly increasing and $G(z)$ is exponentially decreasing on $\Delta$, then $F(z)\cdot G(z)$ is exponentially decreasing on $\Delta$.  
\item The class of exponentially decreasing functions is closed under the classical Fourier transform (see \cite{kawai1970theory}).  The Fourier transform of slowly increasing hyperfunctions will be defined to be dual to this operation via a pairing between slowly increasing hyperfunctions and exponentially decreasing holomorphic functions.
\end{enumerate}

Intuitively, a hyperfunction is slowly increasing when, after fixing the imaginary part of $z$ inside of $i\gamma_j$, its asymptotic growth along the real line is slower than every exponential function.  A hyperfunction is exponentially decreasing on the cone $\gamma$ when the holomorphic functions in a boundary value representation decay exponentially in the real directions which are inside of the cone $\gamma$.  

As previously mentioned, there exists a pairing between slowly increasing hyperfunctions and exponentially decreasing holomorphic functions.  Let $f(x) = F(z + i\gamma 0)$ be a slowly increasing hyperfunction, $G(z)$ an exponentially decreasing analytic function, and $S$ a contour of integration chosen so that $\text{Im}(z) \in i\gamma 0$ for all $z \in S$.  The pairing is given by:
\[ \langle f, G\rangle = \int_{S} F(x+iy)G(x+iy) \,dx \]
Convergence of the integral is guaranteed by the condition that $F(z)G(z)$ is exponentially decreasing.  That the pairing does not depend on the choice of contour follows from the Cauchy integral formula.  The pairing allows us to identify the slowly increasing hyperfunctions as the topological dual space to the space of exponentially decreasing holomorphic functions.  The Fourier transform of a slowly increasing hyperfunction is then defined by a duality with respect to this pairing:
\[ \langle \mathscr{F}(f), G\rangle := \langle f, \mathscr{F}(G) \rangle \]

In practice, the Fourier transform of a hyperfunction is not computed directly from the definition.  Let us now introduce the practical method by which one normally computes the Fourier transform of a slowly increasing hyperfunction.  Suppose that $F(z)$ is a holomorphic function which is exponentially decreasing outside of a closed convex cone $\Delta$.  Letting $z = x+iy$ and $\zeta = \sigma+i\tau$, and suppose that $x \in \Delta$.  We have the following estimate:
\[ \vert \exp(-i\zeta \cdot z) \vert = \exp(y\cdot \sigma + x\cdot \tau) \]
The above estimate shows that $\exp(-i\zeta z)$ will be exponentially decreasing on $\Delta$, so long as we fix $\tau \in -\Delta^{\circ}$.  It then follows that the product $e^{-i\zeta z}F(z)$ is exponentially decreasing on $\mathbb{R}^{n}$.  If $f(x) = F(z+i\gamma 0)$, then its Fourier transform is the hyperfunction given by:
\[ \mathscr{F}(f) = G(\zeta - i\Delta^\circ) = b_{-\Delta^\circ} \left( \int_{S} e^{-i\zeta z}F(z)\,dz \right)\]
This can be extended to an arbitrary boundary value expression $f(x) = \sum_j F_j (z+i\gamma_j 0)$ by linearity, assuming that each of the $F_j(z)$ decreases exponentially outside of some cone.  

We must now deal with the case that $f(x) = F(z + i \gamma 0)$ is a slowly increasing hyperfunction, but that it does not decrease exponentially on any cone.

\begin{define}
Let $\Sigma$ be a finite collection of closed convex cones.  A holomorphic partition of unity is collection of holomorphic functions $\left\{ \chi_\sigma(z)\right\}_{\sigma \in \Sigma}$ such that:
\begin{enumerate}
	\item $\displaystyle{\sum_{\sigma \in \Sigma} \chi_\sigma(z) = 1}$
	\item $\chi_\sigma(z)$ is exponentially decreasing outside of any open cone $\sigma' \supset \sigma$
	\item $\displaystyle{\bigcup_{\sigma \in \Sigma} \sigma = \mathbb{R}^{n}}$
\end{enumerate}
\end{define}

\noindent\textit{Example of a holomorphic partition of unity}:

\vspace{0.2cm}

Let $\Sigma$ denote the collection of orthants in $\mathbb{R}^{n}$.  If $\sigma = (\sigma_1,\dots,\sigma_n)$ is a multi-index whose entries are $\pm 1$ (clearly such objects are in bijection with the orthants), we will denote the corresponding orthant by $\gamma_\sigma$
Consider the following two functions:
\[ \chi_{+}(t) = \frac{1}{1+e^{-t}} \]
\[\chi_{-}(t) = \frac{1}{1+e^{t}} \]
where $t \in \mathbb{C}$ is a complex variable.  We notice that $\chi_+(t)$ is exponentially decreasing on $\text{Re}(t) < 0$ and $\chi_{-}(t)$ is exponentially decreasing on $\text{Re}(t) > 0$.  For a fixed orthant $\sigma \in \Sigma$, define the holomorphic function $\chi_\sigma(z)$ by:
\[ \chi_\sigma(z) = \prod_{i=1}^{n} \frac{1}{1 + e^{\sigma_i z_i}}  \]
This function exponentially decreases on the complement of $\gamma_\sigma$.  The collection $\left\{ \chi_\sigma(z)\right\}_{\sigma \in \Sigma}$ is a holomorphic partition of unity.  

We have introduced holomorphic partitions of unity as an abstract concept, but we will only ever use this example in our computations.  The reason we have done this, as we will see later, is that the computations can be made easier or harder by a clever choice of holomorphic partition of unity (although the actual result of the computation is of course independent of any such choices).  Our main result on the Duistermaat-Heckman hyperfunction of $\Omega SU(2)$ will remain in an integral form, but it is possible that the computation of the Fourier transform could be completed by redoing the computation with a judicious choice of holomorphic partition of unity.  

\vspace{0.2cm}

We are now ready to explain how to compute the Fourier transform of a general slowly increasing hyperfunction.  Again, by linearity of the Fourier transform, we may assume our hyperfunction takes the form $f(x) = F(z+i\gamma 0)$, and that $F(z)$ is a slowly increasing holomorphic function.  Choose a holomorphic partition of unity $\left\{\chi_\sigma(z)\right\}_{\sigma \in \Sigma}$, then we observe that:
\[ F(z) = \sum_{\sigma \in \Sigma} F(z)\chi_\sigma(z) \]
where now, $F(z)\chi_\sigma(z)$ is exponentially decreasing outside of $\sigma$.  By our previous observations,
\begin{equation}
\label{eqn:hyperft}
\mathscr{F}(f) = \sum_{\sigma \in \Sigma}  b_{-\sigma^\circ} \left( \int_{S} e^{-i\zeta z}F(z) \chi_\sigma(z) \,dz \right)
\end{equation}
Equation \ref{eqn:hyperft} exactly tells us how to compute the Fourier transform of a general slowly increasing hyperfunction.


\section{Hyperfunctions arising from localization of Hamiltonian group actions}
\label{HyperHam}

Let $(M,\omega)$ be a finite dimensionial compact symplectic manifold with a Hamiltonian action of a $d$-dimensional compact torus $T$; call the moment map $\mu: M \to \mathfrak{t}^*$.  The symplectic form $\omega$ gives us the Liouville measure $\omega^n / n!$ on $M$, which we we may push forward to $\mathfrak{t}^*$ using the moment map $\mu$.  We let $\mathcal{F}$ denote the connected components of the fixed point set for the $T$ action on $M$; furthermore, if $q \in \mathcal{F}$, we denote by $e_q^T$ the equivariant Euler class of the normal bundle to the fixed point set.  We can identify $e_q^T \in H^*(BT) \simeq \text{Sym}(\mathfrak{t}^*)$ with the product of the weights appearing in the isotropy representation of $T$ on $T_q M$.

\begin{thm} \cite{duistermaat1982variation} 
The measure $\mu_*(\omega^n /n!)$ has a piecewise polynomial density function.  Furthermore, the inverse Fourier transform of $\mu_*(\omega^n/n!)$ has an exact expression:
\begin{equation}
\int_M e^{i \mu(p)(X)} \omega^n/n! = \frac{1}{(2\pi i)^d} \sum_{q \in \mathcal{F}} \frac{e^{i\mu(q)(X)}}{e^T_q(X)}
\label{eqn:localization}
\end{equation}
where $X \in \mathfrak{t}$ is such that $e_{q}^T(X) \neq 0$ for all $q \in \mathcal{F}$.
\end{thm}   

The Duistermaat-Heckman theorem applies to the case where $M$ is finite dimensional and compact.  We are interested in finding some version of a Duistermaat-Heckman distribution in the setting where $M$ is an infinite dimensional manifold with a Hamiltonian group action.  There are some immediate technical obstructions to producing such a distribution.  Most notably, the inability to take a top exterior power of $\omega$ prevents us from defining a suitable Liouville measure.  There are significant analytic challenges in properly defining the left hand side of equation \ref{eqn:localization}; a related problem is defining a rigorous measure of integration for the kinds of path integrals which appear in quantum field theory.  We will not attempt to answer this question in this article.  Nevertheless, it is possible to make sense of the right hand side of Equation \ref{eqn:localization}.

The main goal for this section is to explain how Hamiltonian actions of compact tori yield, in a natural way, hyperfunctions on $\mathfrak{t}$.  The hyperfunction one gets in this way should be a substitute for the the reciprocal of the equivariant Euler class which appears in the localization formula.  We then reinterpret the sum over the fixed points in the localization formula as a hyperfunction on $\mathfrak{t}$, and define the \textit{Duistermaat-Heckman hyperfunction} to be its Fourier transform as a hyperfunction.  

We will start by considering the local picture.  Suppose that $T$ has a Hamiltonian action on a (finite dimensional, for now) complex vector space with weights $\lambda_i$.  Let the weights of the action be given by $W = \left\{\lambda_i\right\}_{i \in I}$.  The weights of the action are linear functionals $\mathfrak{t}_{\mathbb{C}} \to \mathbb{C}$.  For every weight $\lambda \in W$ we get a corresponding half space $H_\lambda = \left\{ y \in \mathfrak{t}\,\vert\, \lambda(y) > 0\right\}$, as well as a hyperfunction:
\[ f_\lambda(x) = \frac{1}{\lambda(z) + i H_\lambda 0} \] 
The singular support of $f_\lambda(x)$ is given by:
\[ SS(f_\lambda) = \left\{ (x,\xi) \in T^*(\mathfrak{t})\,\vert\, \lambda(x) = 0, \exists\, c > 0, \xi = c\, d \lambda(x) \right\} \]

\begin{prop}
If $\mu: V \to \mathfrak{t}^*$ is proper then for all pairs of weights $\lambda, \lambda'$, $SS(f_\lambda) \cap SS(f_{\lambda'})^{\circ} = \emptyset$.
\end{prop}
\begin{proof}
If the moment map is proper, then all of the weights are contained in a half space \cite{guillemin1988kostant}.  There exists $X \in \mathfrak{t}$ such that for any pair of weights $\lambda,\lambda'$ we have both $\lambda(X) > 0$ and $\lambda'(X) > 0$.  Suppose that $(x,\xi) \in SS(f_\lambda) \cap SS(f_\lambda')^{\circ}$.  This means that:
\begin{enumerate}
\item $\lambda(x) = \lambda'(x) = 0$
\item $\exists c,c' > 0$ such that $\xi = c\, d\lambda = -c' d\lambda'$
\end{enumerate}
Rearranging the second condition implies that the function $L = \lambda' + \frac{c}{c'} \lambda$ is constant.  However, we have obtained a contradiction as $L(x) = 0$, while $L(X) > 0$.
\end{proof}

The following is immediate from the proposition.  

\begin{corollary}
\label{corollary:hypereuler}
Let $\displaystyle{\gamma = \bigcap_{\lambda \in W} H_\lambda}$. If $\mu: V \to \mathfrak{t}^*$ is proper, then the following product of hyperfunctions is well defined:
\[ \frac{1}{e^T(x)} = \prod_{\lambda \in W} f_\lambda(x) = b_{\gamma}\left( \prod_{\lambda \in W} \frac{1}{\lambda(z)} \right) \]
\end{corollary}

We can use the reciprocals of the equivariant Euler classes in an expression which imitates the sum over the fixed points in the Duistermaat-Heckman formula.  For any $p \in M^T$, let $W_p = \left\{\lambda^p_j\right\}_{j=1}^N$ denote the set of weights of the isotropy representation of $T$ on $T_p M$.  As in the case of the usual localization formula we must choose a \textit{polarization}, which is simply a choice of vector $\xi \in \mathfrak{t}$ such that for every $p \in M^T$, and for every $\lambda^p_j \in W_p$, we have $\lambda^p_j(\xi) \neq 0$.  For every $\lambda^p_j \in W_p$ we define the \textit{polarized weight} by:
\[ \tilde{\lambda}^p_j = \left\{ \begin{array}{cc} \lambda^p_j & \lambda^p_j(\xi) > 0 \\ -\lambda^p_j & \lambda^p_j(\xi) < 0 \end{array}\right. \]
and we adopt the notation as in \cite{guillemin1988kostant} by setting $(-1)^p = \displaystyle{\prod_{\lambda^p_j \in W_p} \text{sgn}\,\lambda^p_j(\xi)}$.

By definition, for every fixed point $p \in M^T$ we have that the polarized weights are contained in the half plane defined by $\xi$.  We define the cone $\gamma_p = \displaystyle{\bigcap_{\lambda^p_j \in W_p} H_{\tilde{\lambda}^p_j}}$, which is simply the intersection of the half spaces defined by the polarized weights.  We will call:
\[ \frac{1}{e_p^T(x)} = b_{\gamma_p}\left( \prod_{\tilde{\lambda}^p_j \in W_p} \frac{1}{\tilde{\lambda}^p_j(z)} \right) \]
the reciprocal of the equivariant Euler class to the normal bundle of $p$.  

\begin{define}
Suppose that $(M,\omega)$ has a Hamiltonian action of a compact, dimension $d$ torus $T$ such that all the fixed points are isolated; let $M^T$ denote the fixed point set and $\mu: M \to \mathfrak{t}^*$ the moment map.  We will call the following expression the Picken hyperfunction:
\[ L(x) = \frac{1}{(2\pi i)^d}\sum_{p \in M^T}(-1)^p\frac{e^{i\mu(p)(x)}}{e_p^T(x)} \]
\end{define}

\noindent\textit{Example: $S^2$ with a circle action by rotation}

We will first use a simple example to demonstrate that the formalism of hyperfunctions reproduces the results one would expect from the Duistermaat-Heckman function.  We choose a polarization $\xi = -1$.  For the usual Hamiltonian circle action on $S^2$ by counterclockwise rotation about the $z$-axis, there are fixed points at the north and south poles, $N$ and $S$, respectively.  The torus acts on $T_N S^2$ with weight $+1$, while it acts on $T_S S^2$ with weight $-1$.  Let's compute the reciprocal of the equivariant Euler class to the normal bundle of $N$ (as a hyperfunction).  There is only one weight at this fixed point.  We have
\[ \gamma_N = \left\{ x \in i\mathfrak{t}\,\vert\, x < 0\right\} \]
\[ (-1)^N = -1 \] 
The north pole contributes the following hyperfunction as a summand of the Picken hyperfunction, which we denote pictorially in figure \ref{fig:JKS2}:
\[ J_N (x) = b_{\gamma_N}\left( -\frac{e^{iz}}{z} \right) \]

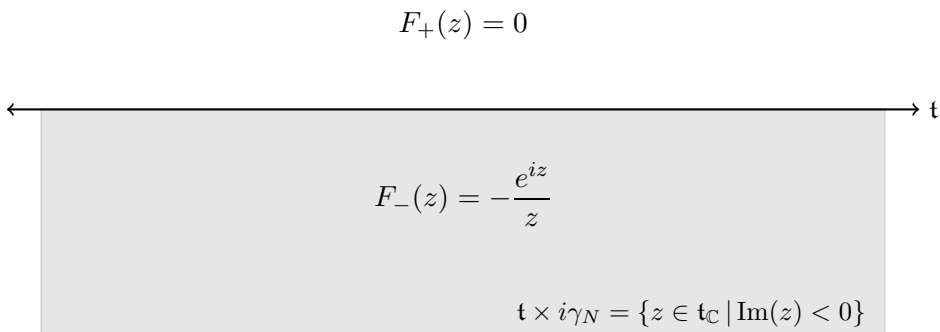
\begin{figure}
\begin{center}
\begin{tikzpicture}[scale = 1.5]
				\draw[black,thick,<->] (-4,0) -- (4,0) node[anchor=west]{$\mathfrak{t}$};
				\draw (2.5,0.2) ;
				
				\draw[black,thick,fill,opacity=0.1] (-3.7,0) rectangle (3.7,-2);
				\draw[black,thick,fill,opacity=0] (-3.7,0) rectangle (3.7,2);
				\draw (2,-1.8) node{$\mathfrak{t}\times i\gamma_N = \left\{z \in \mathfrak{t}_\mathbb{C}\,\vert\, \text{Im}(z) < 0\right\}$};
				
				\draw (0,0.75) node{\large $F_+(z) = 0$};
				\draw (0,-0.75) node{\large $F_-(z) = -\displaystyle{\frac{e^{iz}}{{z}}}$};
\end{tikzpicture}
\caption{A depiction of the hyperfunction $J_N(x) \in \mathcal{B}(\mathfrak{t})$}
\label{fig:JKS2}
\end{center}
\end{figure}

The contribution to the Picken hyperfunction coming from the south pole is computed similarly.  We obtain:
\[ \gamma_S = \left\{ x \in i\mathfrak{t}\,\vert\, x < 0\right\} \]
\[ (-1)^S = -1 \]
and so
\[ J_S(x) = b_{\gamma_S}\left( \frac{e^{-iz}}{z} \right) \] 
Since $\gamma_N = \gamma_S$ in this example, we simply call both of these $\gamma$.  The end result is that the Picken hyperfunction of this Hamiltonian group action is:
\[ 2\pi i \, L(x) = b_{\gamma}\left( -\frac{e^{iz}}{z} + \frac{e^{-iz}}{-z} \right)\]
or, thinking of hyperfunctions on $\mathbb{R}$ as pairs of holomorphic functions, this corresponds to the pair 
\[ 2\pi i L(x) = \left[0, \frac{-e^{iz} + e^{-iz}}{z}\right] \]
Had one chosen the polarization $\tilde{\xi} = +1$, one would have alternatively obtained the presentation
\[ 2 \pi i \tilde{L}(x) = \left[ \frac{e^{iz} - e^{-iz}}{z}, 0 \right] \]
however, $L(x) = \tilde{L}(x)$ as hyperfunctions because their difference extends analytically across the real axis.  The observation here is that a choice of polarization is simply enabling us to write down a presentation of a hyperfunction using a specific set of cones.

The Duistermaat-Heckman hyperfunction is the Fourier transform of the Picken hyperfunction.  We will now compute it according to the formula in equation \ref{eqn:hyperft}.  We choose the holomorphic partition of unity given by the functions:
\[ \chi_{+}(z) = \frac{1}{1+e^{-z}} \]
\[\chi_{-}(z) = \frac{1}{1+e^{z}} \]
which gives a decomposition of the Picken hyperfunction into four parts.
\[ 2\pi i\, L(x) = b_{\gamma}\left( -\frac{e^{iz}}{z} \chi_+(z) \right) + b_{\gamma}\left( \frac{e^{-iz}}{z}\chi_+(z) \right) + b_{\gamma}\left( -\frac{e^{iz}}{z}\chi_-(z) \right) + b_{\gamma}\left( \frac{e^{-iz}}{z}\chi_-(z) \right) \]
The Fourier transform can now be computed termwise, noticing that the first two terms in the above expression are exponentially decreasing on $\text{Re}(z) < 0$, while the third and fourth terms are exponentially decreasing on the cone $\text{Re}(z) > 0$.  Let $1\gg \delta > 0$, then we may write the Fourier transform $\mathscr{F}(L(x)) = G_+(\zeta + i0) + G_{-}(\zeta - i0)$ where:
\[ G_{+}(\zeta) = \int_{-\infty-i\delta}^{\infty-i\delta} -\frac{e^{-i(\zeta-1)z}}{z(1+e^{z})}\,dz + \int_{-\infty-i\delta}^{\infty-i\delta} \frac{e^{-i(\zeta+1)z}}{z(1+e^{z})}\,dz \]
\[ G_{-}(\zeta) = \int_{-\infty-i\delta}^{\infty-i\delta} -\frac{e^{-i(\zeta-1)z}}{z(1+e^{-z})}\,dz + \int_{-\infty-i\delta}^{\infty-i\delta} \frac{e^{-i(\zeta+1)z}}{z(1+e^{-z})}\,dz \]
Each of these integrals can be computed by completing to a semicircular contour in the lower half plane and applying the residue theorem (noting that, as the contour is oriented clockwise, we must include an extra minus sign).  The contour we use for the first integral appearing in $G_+(\zeta)$ is depicted in Figure \ref{fig:contour}, along with the locations of the poles.  

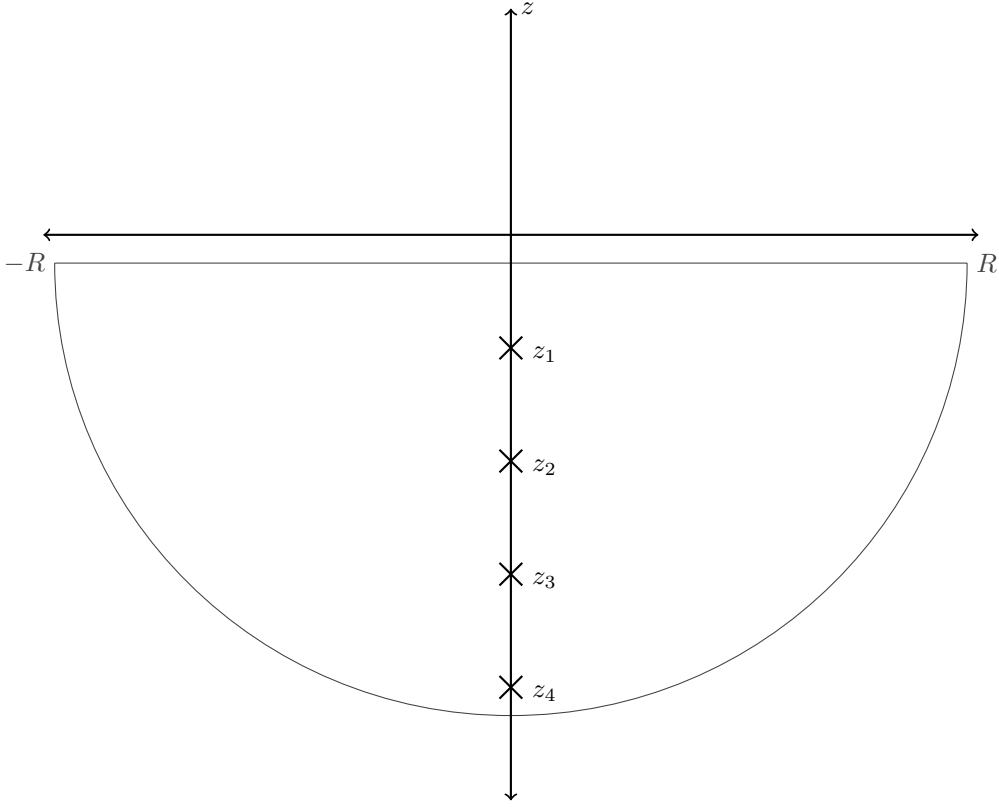
\begin{figure}[h]
\begin{center}
\begin{tikzpicture}[scale = 1.5]
				\draw[black,thick,<->] (-4.1,0) -- (4.1,0);
				\draw[black,thick,<->] (0,-5) -- (0,2) node[anchor = west]{$z$};
				\draw[black,opacity=0.75] (-4,-0.25) node[anchor = east]{$-R$} -- (4,-0.25) node[anchor = west]{$R$};
				\draw[black,opacity=0.75] (4,-0.25) arc(0:-180:4);
				
				\draw[black,thick] (-0.1,-0.9) -- (0.1,-1.1);
				\draw[black,thick] (0.1,-0.9) node[anchor=north west]{$z_1$} -- (-0.1,-1.1);
				
				\draw[black,thick] (-0.1,-1.9) -- (0.1,-2.1);
				\draw[black,thick] (0.1,-1.9) node[anchor=north west]{$z_2$} -- (-0.1,-2.1);
				
				\draw[black,thick] (-0.1,-2.9) -- (0.1,-3.1);
				\draw[black,thick] (0.1,-2.9) node[anchor=north west]{$z_3$} -- (-0.1,-3.1);
				
				\draw[black,thick] (-0.1,-3.9) -- (0.1,-4.1);
				\draw[black,thick] (0.1,-3.9) node[anchor=north west]{$z_4$} -- (-0.1,-4.1);
\end{tikzpicture}
\caption{Integration contour for the first integral in $G_+(\zeta)$}
\label{fig:contour}
\end{center}
\end{figure}

We show how to compute the first integral in the expression for $G_+(\zeta)$; the rest are similar.  The integrand of the first integral in $G_+(\zeta)$ has poles at $z_0 = 0$ and $z_k = -(2k+1)\pi i$ for $k \in \mathbb{Z}$, however, the only poles inside the contour (in the limit as the radius of the semicircle tends to infinity) are the poles at $z_k$ for $k \geq 0$.  Also, in the limit as the radius of the semicircle gets large we see that the contribution to the integral coming from the semicircular part of the contour vanishes because the integrand is exponentially decreasing in $\text{Re}(z)$, and decreasing exponentially in $\text{Im}(z)$ when $\text{Im}(z) < 0$.  By the residue theorem:
\begin{eqnarray*}
\int_{-\infty-i\delta}^{\infty-i\delta} \frac{e^{-i(\zeta-1)z}}{z(1+e^{z})}\,dz &=& -2\pi i \sum_{k=0}^{\infty} \text{Res}\left(\frac{e^{-i(\zeta-1)z}}{z(1+e^{z})}, z = z_k\right) \\
&=& -2\pi i \sum_{k=0}^{\infty} \frac{ e^{-(\zeta - 1)(2k+1)\pi }}{ -(2k+1)\pi i (-1) } \\
&=& 2\pi \sum_{k=0}^{\infty} \int_{c}^{\zeta} e^{-(\zeta' - 1)(2k+1)\pi}\,d\zeta' \\
&=& 2\pi \int_{c}^{\zeta} \sum_{k=0}^{\infty} e^{-(\zeta' - 1)(2k+1)\pi}\,d\zeta' \\
&=& 2\pi \int_{c}^{\zeta} \frac{d\zeta' }{e^{\pi(\zeta'-1)} - e^{-\pi(\zeta'-1)}}\\
&=& \pi \int_{c}^{\zeta} \frac{d\zeta' }{\sinh(\pi(\zeta'-1))} \,d\zeta' \\
&=& \text{Log}\left(\tanh\left( \frac{\pi(\zeta-1)}{2} \right)\right) \quad\text{valid for $\text{Im}(\zeta) > 0$}
\end{eqnarray*}
From the second to third line, we found a primitive function for the summand. From the third to the fourth line we applied the monotone convergence theorem to interchange the order of summation and integration.  A nearly identical computation yields the result:
\[ \int_{-\infty- i\delta}^{\infty-i\delta} \frac{e^{-i(\zeta+1)z}}{z(1+e^{z})}\,dz = \text{Log}\left(\tanh\left( \frac{\pi(\zeta+1)}{2} \right)\right) \quad\text{valid for $\text{Im}(\zeta) > 0$}\]
Summarizing, to this point we have computed:
\[ G_{+}(\zeta) = \text{Log}\left(\tanh\left( \frac{\pi(\zeta+1)}{2} \right)\right)-\text{Log}\left(\tanh\left( \frac{\pi(\zeta-1)}{2} \right)\right) \quad\text{valid for $\text{Im}(\zeta) > 0$}\]
\[ G_{-}(\zeta) = -\text{Log}\left(\tanh\left( \frac{\pi(\zeta+1)}{2} \right)\right)+\text{Log}\left(\tanh\left( \frac{\pi(\zeta-1)}{2} \right)\right)\quad\text{valid for $\text{Im}(\zeta) < 0$} \]
The above expressions can be simplified.  We notice that the holomorphic function $\text{Log}(\tanh(\zeta)) - \text{Log}(\zeta)$ admits an analytic extension across a neighbourhood of the real axis, and is therefore zero as a hyperfunction.  This means that all of the $\tanh$ factors may be ignored for the purposes of computing the hyperfunction Fourier transform.  Therefore, the final result of our computation is:
\[ \mathscr{F}(L(x)) = b_{+}\left( -\frac{1}{2\pi i}\text{Log}\left(\frac{\zeta-1}{\zeta+1} \right) \right) - b_{-}\left( -\frac{1}{2\pi i}\text{Log}\left(\frac{\zeta-1}{\zeta+1} \right) \right) \]
which we recognize as the standard defining hyperfunction of $\chi_{[-1,1]}(x)$ (see \cite{kaneko1989introduction} Example 1.3.11, p. 29).  This has shown that the Fourier transform of the Picken hyperfunction gives the standard defining hyperfunction of the Duistermaat-Heckman distribution.

Jeffrey and Kirwan, building on work of Witten \cite{witten1992two}, formalized the notion of a residue in symplectic geometry \cite{jeffrey1995localization}.  They fruitfully applied this construction to compute relations in the cohomology ring of the moduli space of stable holomorphic bundles on a Riemann surface \cite{jeffrey1995intersection}.  We expect that the properties that uniquely characterize the residue (c.f. Proposition 8.11, \cite{jeffrey1995localization}) can be recovered from the usual notion of a residue \cite{griffiths2014principles} of a multivariable complex meromorphic function using our construction of the Picken hyperfunction.


\section{$\Omega G$ and its Hamiltonian group action}

Let $G$ be a compact connected real Lie group, and call its Lie algebra $\mathfrak{g}$.  In this article we will consider the space of smooth loops in $LG = C^\infty(S^1,G)$.  $LG$ is itself an infinite dimensional Lie group, with the group operation taken to be multiplication in $G$ pointwise along a loop.  The Lie algebra of $LG$ is easily seen to consist of the space of smooth loops into the Lie algebra, which we denote $L\mathfrak{g}$.  

We will also consider its quotient $\Omega G = LG/G$, where the quotient is taken with respect to the subgroup of constant loops.  One may alternatively identify $\Omega G$ as the collection of loops, such that the identity in $S^1$ maps to the identity in $G$: 
\[  \Omega G = \left\{ \gamma \in LG: \,\gamma(1) = e\right\} \]
Its Lie algebra can be identified with the subset $\Omega\mathfrak{g} = \left\{X: S^1 \to \mathfrak{g}\,\vert\, X(0) = 0\right\}$.

$\Omega G$ has a lot of extra structure, which essentially comes from its realization as a coadjoint orbit of a central extension of $LG$ \cite{khesin2008geometry}.  We can give $\Omega G$ a symplectic structure as follows.  Since $G$ is a compact Lie group, there exists a non-degenerate symmetric bilinear form $\langle\cdot,\cdot\rangle: \mathfrak{g}\times\mathfrak{g} \to \mathbb{R}$.  This form induces an antisymmetric form:
\[ \omega_e: L\mathfrak{g} \times L\mathfrak{g} \to \mathbb{R} \]
\[  (X,Y) \mapsto \frac{1}{2\pi} \int_{0}^{2\pi} \langle X(\theta), Y'(\theta) \rangle\,d\theta \]
This bilinear form is an antisymmetric, non-degenerate form when restricted to $\Omega\mathfrak{g}$, and extends to a symplectic form on $\Omega G$ using a left trivialization of the tangent bundle of $\Omega G$.  That is, for every $\gamma \in \Omega G$ we fix the isomorphism 
\[ T_\gamma \Omega G \simeq \Omega\mathfrak{g} \]
\[  X \mapsto \left( \theta \mapsto \gamma^{-1}(\theta)X(\theta) \right) \]
This choice allows us to define a form on $\Omega G$ as:
\[ \omega_\gamma: T_\gamma \Omega G \times T_\gamma \Omega G \to \mathbb{R} \]
\[ (X,Y) \mapsto \omega_e(\gamma^{-1}X,\gamma^{-1}Y) \]
The form so defined is symplectic; a proof can be found in \cite{pressleysegal}.  

Consider the following group action on $\Omega G$. Fix $T\subseteq G$ a maximal compact torus, and let $\mathfrak{t}$ be its Lie algebra.  Pointwise conjugation by elements of $T$ defines a $T$ action on $\Omega G$.
\[ T\times \Omega G \to \Omega G \]
\[ t\cdot \gamma = \left( \theta \mapsto t\gamma(\theta)t^{-1} \right) \]
There is also an auxiliary action of $S^1$ on $\Omega G$, which comes about by descending the loop rotation action on $LG$ to the quotient $LG/G$.  Explicitly,
\[ S^1 \times \Omega G \to \Omega G \]
\[ \exp(i\psi) \cdot \gamma = \left( \theta \mapsto \gamma(\theta+\psi)\gamma(\psi)^{-1} \right) \] 
These actions commute with one another, so define an action of $T\times S^{1}$ on $\Omega G$.  We will let $\text{pr}_{\mathfrak{t}}: \mathfrak{g} \to \mathfrak{t}$ denote the orthogonal projection coming from the Cartan-Killing form.  We now define two functions on $\Omega G$:
\[ p(\gamma) = \frac{1}{2\pi}\text{pr}_{\mathfrak{t}}\left(\int_0^{2\pi} \gamma^{-1}(\theta)\gamma'(\theta)\,d\theta \right) \]
\[ E(\gamma) = \frac{1}{2\pi}\int_0^{2\pi} \vert\vert \gamma'(\theta)\vert\vert^{2} \,d\theta \]

\begin{prop} \cite{atiyahpressley}
The $T\times S^1$ action on $\Omega G$ is Hamiltonian.  The moment map is given by:
\[ \mu: \Omega G \to \text{Lie}(T\times S^{1}) \]
\[ \gamma \mapsto \left(\begin{array}{c} p(\gamma) \\ E(\gamma) \end{array}\right) \]
Furthermore, the Hamiltonian vector fields associated to the group action are given by:
\[(X_{E})_\gamma = \gamma'(\theta) - \gamma(\theta)\gamma'(0) \]
\[ (X_{p_\tau})_\gamma = \tau\gamma(\theta) - \gamma(\theta)\tau \]
where $\tau \in \mathfrak{t}$.
\end{prop}

If $\beta \in \mathfrak{t}\oplus\mathbb{R}$ then we let $(X_\beta)_\gamma$ denote the Hamiltonian vector field evalutated at the loop $\gamma$.


\section{Fixed Points Sets of Rank One Subtori}

We will now proceed to identify the fixed point sets of dimension one subtori of $T\times S^{1}$ acting on $\Omega G$.  The moment map image of the fixed point submanifolds should correspond to the locus where the Duistermaat-Heckman density function is not differentiable.  Using the exponential map, we identify $X_*(T\times S^1) \simeq P\times\mathbb{Z}$, where $P$ is the coweight lattice of $\mathrm{Lie}(T)$.  Fix an element $\beta = (\lambda, m) \in X_{*}(T\times S^{1})$ and call the cocharacter it generates by $T_{\beta}$.  Let $\Lambda \in X_*(T)$ be the cocharacter generated by $\lambda$,
\[ \Lambda(\theta) = \exp(i\lambda\theta) \]
We will say the fixed point set of $T_{\beta}$ is trivial when $\Omega G^{T_{\beta}} = \text{Hom}_{\text{Grp}}(S^{1},T)$.  In this section, we say that $L \subseteq G$ is a \textit{Levi subgroup} if and only if there exists a parabolic subgroup $Q \subseteq G_\mathbb{C}$ such that $L_\mathbb{C}$ is a Levi factor of $Q$.  Every Levi subgroup of $G$ is the centralizer of a subtorus$S\subseteq T$.  

If we have two groups $K$ and $N$, together with a map $\varphi: K \to \text{Aut}\,N$, then we can construct the semidirect product group $N\rtimes K$ whose point set is the Cartesian product $N\times K$, but the group operation is $(n,k)\cdot(n',k') = ((\phi(k')\cdot n)n',kk')$.  In our specific context, if we fix any Levi subgroup $L\subseteq G$, we can construct a group homomorphism: 
\[ \varphi_{\beta}: S^{1}\to \text{Aut}\,L\] 
\[ \varphi_{\beta}(\psi)\cdot x =  \Lambda\left(\frac \psi m\right)^{-1} x\, \Lambda\left(\frac\psi m\right)\]
\textit{Remarks}:
\begin{enumerate}
\item Since $\varphi_\beta(1) = \mathrm{id}_L$ and $S^1$ is connected then we may consider $\varphi_\beta: S^1 \to \mathrm{Inn}(L)$.  We identify $\mathrm{Inn}(L) \simeq L_\text{ad}$, which may further be identified with $[L,L]/ Z(L) \cap [L,L]$.  Under these identifications, $\varphi_\beta \in X_* (T_{\text{ad}})$ is a cocharacter of the maximal torus in $L_\text{ad}$.   \item This homomorphism is well defined if and only if $\Lambda(\frac{2\pi}{m}) \in Z(L)$.  In particular, $\lambda / m$ must be an element of the coweight lattice for the Levi subgroup $L$, mod $\mathfrak{z}(L)$.
\item $\varphi_\beta = \varphi_{\beta'}$ if and only if $\lambda/m - \lambda'/m' \in \mathfrak{z}(L)$
\end{enumerate}

We will denote the resulting semidirect product group as $L\rtimes_{\beta} S^{1}$.

\vspace{0.2cm}

It can be easily seen that for any Levi subgroup $L$, $T\times S^{1}$ is a maximal torus of $L\rtimes_{\beta}S^{1}$.  Any one parameter subgroup of $L\rtimes_{\beta}S^{1}$ is abelian, and is therefore contained in a maximal torus conjugate to $T\times S^{1}$.  We can obtain all one parameter subgroups by considering one of the form $(\eta(\theta),\theta)$ for $\eta \in \mathrm{Hom}(S^{1},T)$, then conjugating by an element of $L\rtimes_{\beta} S^{1}$.  
\begin{equation}
\label{equation:loopform}
\gamma(\theta) = \Lambda\left(\frac{\psi-\theta}{m}\right)g\Lambda\left(\frac{\theta}{m}\right)\eta(\theta)g^{-1}\Lambda\left(\frac{\psi}{m}\right)^{-1}
\end{equation}


\begin{prop}
\label{prop:BetaFixed}
For any $\beta \in P\times \mathbb{Z}$, there exists a Levi subgroup $T \subseteq L_\beta \subseteq G$, such that $\gamma \in \Omega G^{T_\beta}$ if and only if $(\gamma(\theta),\theta)$ is a one parameter subgroup of $L_\beta\rtimes_{\beta} S^{1}$.
\end{prop} 
\begin{proof}

Fix $\beta \in P\times \mathbb{Z}$ and set $L_\beta = Z_G(\Lambda(2\pi/m))$; that $T\subseteq L_\beta$ follows, since $\Lambda(2\pi/m) \in T$ and $T$ is abelian.  

Suppose we have a loop $\gamma$ fixed by $T_{\beta}$.  Recall how $T_{\beta}$ acts on a loop $\gamma \in \Omega G$.  For every $(\Lambda(\psi),\exp(im\psi)) \in T_{\beta}$, the action is:
\[ (\Lambda(\psi),\exp(im\psi))\cdot \gamma(\theta) = \Lambda(\psi)\gamma(\theta + m\psi)\Lambda^{-1}(\psi)\gamma(m\psi)^{-1} \qquad \forall\,\psi,\theta \in [0,2\pi)\]
Let's rescale the $\psi$ variable, then by periodicity we may write the condition to be fixed under $T_{\beta}$ as:
\[ \gamma(\theta + \psi) = \Lambda\left(\frac \psi m\right)^{-1}\gamma(\theta)\Lambda\left(\frac\psi m\right)\gamma(\psi)\qquad\forall\,\theta,\psi\in [0,2\pi) \]
When $\psi = 2\pi$ in the above equation we get the condition $\gamma(\theta) \in L_\beta$ for all $\theta$.  That $(\gamma(\theta),\theta)$ is a one parameter subgroup of $L\rtimes_{\beta} S^{1}$ follows immediately from the multiplication rule for the semidirect product.  

Now suppose conversely that $(\gamma(\theta),\theta)$ is a one parameter subgroup of $L\rtimes_{\beta}S^{1}$.  There exists $\eta \in X_*(T)$, $g \in L$ and $\psi \in S^1$ such that $\gamma$ can be written as in equation \ref{equation:loopform}.  To show that $\gamma$ is fixed by $T_\beta$ it suffices to prove that the Hamiltonian vector field corresponding to $\beta$ vanishes at $\gamma$.  This is a straightforward (but tedious) verification.

\end{proof} 

A consequence of the previous proposition is that for any such $\beta$, there exists a Levi subgroup $L_\beta$ such that $\Omega G^{T_\beta} = \Omega L_\beta^{T_\beta}$.  This follows, since the semidirect product formula forces any loop fixed under $T_\beta$ to have its image be contained in $L_\beta$.

\vspace{0.2cm}

\begin{prop}
\label{prop:fps}
Every connected component of the fixed point set of $T_{\beta}$ is a translate of an adjoint orbit in $\text{Lie}(L_\beta)\subseteq \mathfrak{g}$.
\end{prop}
\begin{proof}
Fix a loop $\gamma$ in some connected component of the fixed point set of $T_{\beta}$.  Using the exponential map on $L\rtimes_{\beta}S^{1}$, it can be seen that $(\gamma(\theta),\theta)$ is a one-parameter subgroup of $L_\beta \rtimes_{\beta}S^{1}$ if and only if $\gamma$ is a solution to the differential equation:
\[ \frac{d\gamma}{d\theta} = \left[\gamma(\theta),\frac{\lambda}{m}\right] + \gamma(\theta)\gamma'(0) \]
Compactness of $G$ (and therefore, of $L_\beta$, since it is a closed subgroup) and the Picard-Lindel\"of theorem allow us to identify the loops in the fixed point set of $T_{\beta}$ with their initial conditions $\gamma'(0) \in \mathfrak{g}$.  We can use equation \ref{equation:loopform} to compute $\gamma'(0)$:
\[ \gamma'(0) = \mathrm{Ad}_{\Lambda(\frac{\psi}{m})g}\left[\frac{\lambda}{m} + \eta'(0)\right] - \frac{\lambda}{m} \]
Any other loop in the same connected component of the fixed point set of $T_{\beta}$ can be obtained by varying $g \in L_\beta$ and $\psi \in [0,2\pi)$.
\end{proof}

Notice that by fixing $\lambda = 0$ in the preceeding discussion, we recover the result that the fixed point set of the loop rotation action consists of the group homomorphisms $S^1 \to G$ \cite{pressleysegal}.

The last result of this section characterizes exactly when two rank one subtori have the same fixed point sets.  

\begin{prop}
\label{prop:fpsequal}
Let $\beta = (\lambda,m)$ and $\beta' = (\lambda', m')$ be generators of rank one subgroups $T_\beta$, $T_{\beta'}$ of $T\times S^1$, and let $L_\beta, L_{\beta'}$ be the Levi subgroups provided by Proposition \ref{prop:BetaFixed}.  Then, $\Omega G^{T_\beta} = \Omega G^{T_{\beta'}}$ if and only if $\lambda/m - \lambda'/m' \in \mathfrak{z}(L_\beta)$
\end{prop}

\noindent\textit{Remark}: If $\lambda/m - \lambda'/m' \in \mathfrak{z}(L_\beta)$ then $L_\beta = L_{\beta'}$.  This is due to the fact that $L_\beta$ was defined to be the $G$-centralizer of $\exp(2\pi i \lambda/m)$ (and similarly for $L_{\beta'}$).  

\begin{proof}
Suppose that the fixed point sets of $T_\beta$ and $T_{\beta'}$ are equal.  Then for any $\gamma$, we have $(X_\beta)_\gamma = 0$ if and only if $(X_{\beta'})_\gamma = 0$.  These conditions yield two differential equations:
\[ 0 = m\frac{d\gamma}{d\theta} - \gamma(\theta)\gamma'(0) + \lambda \gamma(\theta) - \gamma(\theta)\lambda \]
\[ 0 = m'\frac{d\gamma}{d\theta} - \gamma(\theta)\gamma'(0) + \lambda' \gamma(\theta) - \gamma(\theta)\lambda' \]
We may subtract these, and left translate back to $\mathfrak{g}$ to get the condition:
\[ \forall\,\gamma \in \Omega G^{T_\beta}, \theta \in [0,2\pi), \quad \mathrm{Ad}_{\gamma(\theta)}\left( \frac\lambda m - \frac{\lambda'}{m'} \right) = \frac\lambda m - \frac{\lambda'}{m'} \]
The derivative of this condition at the identity is
\[ \left[\gamma'(0), \frac\lambda m - \frac{\lambda'}{m'} \right] = 0 \]
so the statement is proved if for every element $Y \in \text{Lie}([L_\beta,L_\beta])$, there exists $\gamma \in \Omega G^{T_\beta}$ and $c \in \mathbb{R}$ such that $Y = c \gamma'(0)$.  By Proposition \ref{prop:fps}, we can identify the set of all such $\gamma'(0)$ with a translated adjoint orbit.  This can be achieved by choosing a cocharacter $\eta(\theta)$ such that $\eta'(0) + \frac{\lambda}{m}$ is regular for the $\mathrm{Ad}_{L_\beta}$-action and $\eta'(0)$ is sufficiently large so that the translated adjoint orbit intersects every ray through the origin.

Conversely, if $\lambda/m - \lambda'/m' \in \mathfrak{z}(L_\beta)$ then by the above remark, $L_\beta = L_\beta'$, and furthermore, $\beta$ and $\beta'$ yield identical automorphisms $\varphi_\beta = \varphi_{\beta'}: S^1 \to \mathrm{Aut}(L_\beta)$.  Then by Proposition \ref{prop:BetaFixed} we have $\Omega G^{T_\beta} = \Omega G^{T_{\beta'}}$.  
\end{proof}


\section{An explicit example: The loop space of $SU(2)$}

When $G = SU(2)$, the general theory of the previous section can be understood in a very explicit way.  The way to do this is to translate the condition of being fixed under the group action into a solution of a system of differential equations for the matrix parameters.  Let's work through this derivation.  We can describe an element $\gamma(t) \in \Omega SU(2)$ by:
\[ \gamma(t) = \left( \begin{array}{cc} \alpha(t) & -\beta(t)^{*} \\ \beta(t) & \alpha(t)^{*} \end{array}\right) \]
Subject to the constraints $\vert \alpha(t)\vert^{2} + \vert \beta(t)\vert^{2} = 1$ for all $t \in [0,2\pi]$, $\alpha(0) = 1$, and $\beta(0) = 0$.  One-parameter subgroups correspond bijectively with elements of the Lie algebra of $T\times S^{1}$.  In that spirit, fix some element $(\theta,\psi) \in \mathfrak{t}\oplus \mathbb{R}$, exponentiate to the group, and act on our loop $\gamma(t)$
\begin{eqnarray*}
\left( \left( \begin{array}{cc} e^{i\theta} &0  \\ 0 & e^{-i\theta} \end{array}\right),e^{i\psi}\right)\cdot\gamma(t) &=& e^{i\psi}\cdot\left( \left( \begin{array}{cc} e^{i\theta} &0  \\ 0 & e^{-i\theta} \end{array}\right)\left( \begin{array}{cc} \alpha(t) & -\beta(t)^{*} \\ \beta(t) & \alpha(t)^{*} \end{array}\right)\left( \begin{array}{cc} e^{-i\theta} &0  \\ 0 & e^{i\theta} \end{array}\right)\right) \\
&=& e^{i\psi}\cdot \left( \begin{array}{cc} \alpha(t) & -e^{i2\theta}\beta(t)^{*} \\ e^{-i2\theta}\beta(t) & \alpha(t)^{*} \end{array}\right) \\
&=& \left( \begin{array}{cc} \alpha(t+\psi) & -e^{i2\theta}\beta(t+\psi)^{*} \\ e^{-i2\theta}\beta(t+\psi) & \alpha(t+\psi)^{*} \end{array}\right)\left( \begin{array}{cc} \alpha(\psi)^{*} & e^{i2\theta}\beta(\psi)^{*} \\ -e^{-i2\theta}\beta(\psi) & \alpha(\psi) \end{array}\right) \\
&=& \left( \begin{array}{cc} \alpha(t) & -\beta(t)^{*} \\ \beta(t) & \alpha(t)^{*} \end{array}\right) \quad\text{when $\gamma(t)$ is a fixed loop}
\end{eqnarray*}
so by rearranging slightly
\begin{eqnarray*}
\left( \begin{array}{cc} \alpha(t+\psi) & -e^{i2\theta}\beta(t+\psi)^{*} \\ e^{-i2\theta}\beta(t+\psi) & \alpha(t+\psi)^{*} \end{array}\right) &=& \left( \begin{array}{cc} \alpha(t) & -\beta(t)^{*} \\ \beta(t) & \alpha(t)^{*} \end{array}\right)\left( \begin{array}{cc} \alpha(\psi) & -e^{i2\theta}\beta(\psi)^{*} \\ e^{-i2\theta}\beta(\psi) & \alpha(\psi)^{*} \end{array}\right) \\
&=& \left( \begin{array}{cc} \alpha(t)\alpha(\psi) - e^{-i2\theta}\beta(t)^{*}\beta(\psi) & -\alpha(\psi)^{*}\beta(t)^{*}-e^{i2\theta}\alpha(t)\beta(\psi)^{*} \\ \alpha(\psi)\beta(t)+e^{-i2\theta}\alpha(t)^{*}\beta(\psi) & \alpha(t)^{*}\alpha(\psi)^{*} - e^{i2\theta}\beta(t)\beta(\psi)^{*} \end{array}\right)
\end{eqnarray*}
this yields the finite difference relations:
\[ \alpha(t+\psi) = \alpha(t)\alpha(\psi) - e^{-i2\theta}\beta(t)^{*}\beta(\psi) \]
\[ \beta(t+\psi) = e^{i2\theta}\alpha(\psi)\beta(t)+\alpha(t)^{*}\beta(\psi)\]
We use these infinitesimal form of these relations to get the necessary system of differential equations.  Set $\theta = ns$ and $\psi = ms$ so that we can vary the group element along a fixed one parameter subgroup.  
\begin{eqnarray*}
m \frac{d\alpha}{dt} &=& \lim_{s\to 0} \frac{\alpha(t+ms)-\alpha(t)}{s} \\
&=& \lim_{s\to 0} \frac{\alpha(t)\alpha(ms)-e^{-i2ns}\beta(t)^{*}\beta(ms)-\alpha(t)}{s} \\
&=& \alpha(t) \lim_{s\to 0}\frac{\alpha(ms) - 1}{s} - \beta(t)^{*}\lim_{s\to 0}  \frac{e^{-i2ns}\beta(ms)}{s} \\
&=& m \alpha(t)\alpha'(0) - m\beta(t)^{*}\beta'(0)
\end{eqnarray*}
And similarly for $\beta(t)$,
\begin{eqnarray*}
m \frac{d\beta}{dt} &=& \lim_{s\to 0} \frac{\beta(t+ms)-\beta(t)}{s} \\
&=& \beta(t) \lim_{s\to 0}\frac{e^{i2ns}\alpha(ms) - 1}{s} + m\alpha(t)^{*}\beta'(0) \\
&=& \beta(t)\left[ i2n e^{i2ns}\alpha(ms) + me^{i2ns}\alpha'(ms)\right]\bigg|_{s=0} + m\beta'(0)\alpha(t)^{*} \\
&=& (i2n+m\alpha'(0))\beta(t) + m\beta'(0)\alpha(t)^{*}
\end{eqnarray*}
so the system of differential equations we must solve (for $m\neq 0$, when $m=0$ the problem is trivial) is given by:
\[ \frac{d\alpha}{dt} = \alpha(t)\alpha'(0) - \beta(t)^{*}\beta'(0)\]
\[ \frac{d\beta}{dt} = \beta'(0)\alpha(t)^{*} + \left(i2\frac nm + \alpha'(0)\right)\beta(t)\]
These differential equations are exactly the ones we could have gotten by searching for zeroes of the Hamiltonian vector field corresponding to $(n,m) \in \mathfrak{t}\oplus\mathbb{R}$ (c.f. the differential equation given in Proposition \ref{prop:fps}).  The system we have described depends on four parameters: $n$, $m$, $\alpha'(0)$ and $\beta'(0)$.  Once we fix these parameters, the solutions $\alpha(t)$ and $\beta(t)$ are uniquely determined.  The parameters $n$ and $m$ are fixed from the start, so are only free to vary $\alpha'(0)$ and $\beta'(0)$.  The choices that will turn out to yield periodic solutions will be exactly those loops whose derivatives at the identity are elements of the translated adjoint orbits of Proposition \ref{prop:fps}.

An explicit analytic solution to the system of differential equations can be found by expanding $\alpha(t)$ and $\beta(t)$ in Fourier series.
\[ \alpha(t) = \sum_{k=-\infty}^{\infty} \alpha_{k}e^{-ikt} \]
\[ \beta(t) = \sum_{k=-\infty}^{\infty} \beta_{k}e^{-ikt} \]
Plugging these expressions into the system of differential equations yields a system of algebraic relations for each $k$:
\begin{eqnarray}
0 &=& (\alpha'(0)+ik)\alpha_{k} - \beta'(0)\beta_{-k}^{*} \\
0 &=& \beta'(0)\alpha_{-k}^{*} + \left(i(2\frac nm + k) + \alpha'(0)\right)\beta_{k}
\end{eqnarray}
We can solve by taking $ik-i2\frac nm+\alpha'(0)^{*}$ times the first equation above and substituting into the conjugate of the second equation (replacing $k$ by $-k$).  For each $k$, this yields the expression:
\[ \left(\vert \alpha'(0)\vert^{2} + \vert \beta'(0)\vert^{2} -k^{2} +\frac{2n}{m}\alpha'(0) + \frac{2n}{m}k\right)\alpha_{k} = 0 \]
which implies that either $\alpha_{k} = 0$ or (after completing the square and setting $\alpha'(0) = iA$, which is necessary for $\gamma \in \Omega SU(2)$):
\begin{equation} 
\label{eqn:su2crit}
\left(k-\frac nm\right)^{2} = \left(A +\frac nm\right)^{2} + \vert \beta'(0)\vert^{2}
\end{equation}

The purpose of equation \ref{eqn:su2crit} is to characterize the set of initial conditions for the differential equations above which yield periodic solutions; in other words, equation \ref{eqn:su2crit} exactly identifies to fixed point set of the subtorus generated by $(n,m)$ with a disjoint union of translated adjoint orbits of $SU(2)$, as in Proposition \ref{prop:fps}.  It is evident from equation \ref{eqn:su2crit} that for any loop fixed under the subgroup $(n,m)$ at most two Fourier modes can be non-zero.  These two modes correspond to precisely the values of $k$ that satisfy $k-\frac nm = \pm C$ for some constant $C$, for which we require integer solutions of $k$.  We can get two distinct solutions only if $n+Cm = ml$ and $n- Cm = ml'$, which implies that $C = (l-l')/2$ is a half integer and $n/m = (l+l')/2$ is a half integer.  

We should contextualize this result in the language of Proposition \ref{prop:BetaFixed}.  For $SU(2)$ only two Levi subgroups are possible: the maximal torus $T$, or $SU(2)$ itself.  The former case arises when $n/m \notin \frac12 \mathbb{Z}$, and the latter case arises when $n/m \in \frac12 \mathbb{Z}$.  Stated slightly differently, when $n/m \in P^\vee \subseteq \mathfrak{t}$ is in the coweight lattice of $SU(2)$, then $\exp(2\pi i n/m) \in Z(SU(2))$ and the Levi subgroup corresponding to $(n,m)$ is $G = SU(2)$ (and is the maximal torus otherwise).


\section{Isotropy Representation of $T\times S^1$}

Whenever a group $G$ acts on a manifold $M$ and $x \in M$ is a fixed point of the action, one obtains a representation of $G$ on $T_x M$ by taking the derivative of the action map at $x$.  In this section, we compute this representation on the tangent space at any fixed point of the $T\times S^1$ action on $\Omega G$.  As we are considering the action of torus on a vector space, we present a splitting of the representation in terms of its weight vectors.

\begin{prop}
\label{prop:isotropyrep}
Let $\gamma$ be fixed by $T\times S^{1}$ and suppose that $(t,\psi) \in T\times S^{1}$, then after identifying $T_{\gamma}\Omega G \simeq \Omega\mathfrak{g}$, the isotropy representation of $T\times S^{1}$ on $T_{\gamma}\Omega G$ is given by:
\[ (t,e^{i\psi})_{*}: \Omega \mathfrak{g}\to\Omega\mathfrak{g} \]
\[ X(\theta) \mapsto \mathrm{Ad}_{t\gamma(\psi)}\,X(\theta+\psi) \]
\end{prop}

\begin{proof}
By embedding $G$ in $U(n)$ we may assume that $G$ is a matrix group.  Pick any variation $\delta\gamma \in T_{\gamma}\Omega G$ and write $\delta\gamma(\theta) = \gamma(\theta)X(\theta)$ for some $X\in \Omega\mathfrak{g}$.  We compute the pushforward:
\begin{eqnarray*}
(t,e^{i \psi})_{*}(\delta\gamma) &=& \frac{d}{d\epsilon}\bigg|_{\epsilon = 0}\left[ (t,\psi)\cdot(\gamma(\theta)+\epsilon\gamma(\theta)X(\theta))\right] \\
&=& \frac{d}{d\epsilon}\bigg|_{\epsilon = 0}\left[ t(\gamma(\theta+\psi)+\epsilon\gamma(\theta+\psi)X(\theta+\psi))(1+\epsilon X(\psi))^{-1}\gamma(\psi)^{-1}t^{-1}\right] \\
&=& \frac{d}{d\epsilon}\bigg|_{\epsilon = 0}\left[\sum_{j=0}^{\infty} (-1)^{j}\epsilon^{j}t(\gamma(\theta+\psi)+\epsilon\gamma(\theta+\psi)X(\theta+\psi))X(\psi)^{j}\gamma(\psi)^{-1}t^{-1}\right] \\
&=& t\gamma(\theta+\psi)\left[ X(\theta+\psi)-X(\psi)\right]\gamma(\psi)^{-1}t^{-1}
\end{eqnarray*}
But now since $\gamma$ is fixed under $T\times S^{1}$, we have $\gamma(\theta) = (t,\psi)\cdot\gamma(\theta) = t\gamma(\theta+\psi)\gamma(\psi)^{-1}t^{-1}$ which implies $\gamma(\theta)t\gamma(\psi) = t\gamma(\theta+\psi)$.  Plugging in to the last line of the above yields the desired formula for the isotropy representation, noticing that the constant term is equivalent to zero in the quotient $\Omega\mathfrak{g} \simeq L\mathfrak{g}/\mathfrak{g}$.
\end{proof} 

The proposition above allows us to compute a weight basis for the isotropy representation, along with the corresponding weights.

\begin{thm}
\label{thm:weightsofisotropy}
If $\gamma(\theta) = \exp(\eta\theta) \in \Omega G$ is fixed by $T\times S^{1}$ (i.e. $\eta \in Q^\vee$), the $T\times S^{1}$ action on $T_{\gamma}\Omega G$ decomposes into non-trivial irreducible subrepresentations:
\[ T_{\gamma}\Omega G \simeq \Omega\mathfrak{g} \simeq \bigoplus_{k =1}^{\infty} \left( \bigoplus_{\alpha \in R} V_{\alpha,k} \oplus \bigoplus_{i=1}^n V_{i,k} \right) \]
The weight of $T\times S^1$ on $V_{\alpha,k}$ is:
\[ \lambda^k_\alpha: \text{Lie}(T\times S^1)_{\mathbb{C}} \to \mathbb{C} \]
\[ \lambda^k_\alpha(x_1,x_2) = \alpha(x_1 + \eta x_2) + k x_2 \]
A basis of weight vectors for $V_{\alpha,k}$ is:
\[ X^{(1)}_{\alpha, k} = i\sigma^\alpha_y \cos(k\theta) \pm i\sigma^\alpha_x \sin(k\theta) \]
\[ X^{(2)}_{\alpha,k} =  i\sigma^\alpha_x \cos(k\theta) \mp i\sigma^\alpha_y \sin(k\theta) \] 
where the plus or minus sign is taken depending on whether $\alpha$ is a positive or negative root, respectively.
The weight of $T\times S^1$ on $V_{i,k}$ is:
\[ \lambda^k_i: \text{Lie}(T\times S^1)_{\mathbb{C}} \to \mathbb{C} \]
\[ \lambda^k_i(x_1,x_2) = k x_2 \]
A basis of weight vectors for $V_{i,k}$ is given by:
\[ X_{i,k}^{(1)} = i\sigma^\alpha_{z} \cos(k\theta) \]
\[ X_{i,k}^{(1)} = i\sigma^\alpha_{z} \sin(k\theta) \]
\end{thm}
\begin{proof}
We will check that the pair $(X^{(1)}_{\alpha,k},X^{(2)}_{\alpha,k})$ is a weight basis for $V_{\alpha,k}$ with the appropriate weight; the other cases are similar.  Let $t = e^{x_1}$ for $x_1 \in \mathfrak{t}$, let $x_2 \in \text{Lie}(S^1)$, and let $\Theta = \alpha(x_1 + \eta x_2)$.  By Proposition \ref{prop:isotropyrep},
\begin{eqnarray*}
(t,e^{i x_2})_*X^{(1)}_{\alpha,k} &=& \text{Ad}_{t\gamma(x_2)}\left( i\sigma^\alpha_y \cos(k\theta) + i\sigma^\alpha_x \sin(k\theta) \right) \\
&=& i\left( \sigma_y^\alpha \cos\Theta + \sigma_x^\alpha \sin\Theta \right) \cos(k(\theta +x_2)) + i\left( \sigma_x^\alpha \cos\Theta - \sigma_y^\alpha \sin\Theta \right) \sin(k(\theta + x_2)) \\
&=&i\left( \sigma_y^\alpha \cos\Theta + \sigma_x^\alpha \sin\Theta \right)\left( \cos k x_2 \cos k\theta -  \sin kx_2 \sin k\theta\right) \\
&& \quad  +\;i\left( \sigma_x^\alpha \cos\Theta - \sigma_y^\alpha \sin\Theta \right)\left( \sin kx_2 \cos k\theta + \cos kx_2 \sin k\theta\right) \\
&=& i\cos(\Theta + kx_2 )\, \sigma_y^\alpha \cos k \theta + i\sin(\Theta + kx_2 )\, \sigma_x^\alpha \cos k \theta \\
&&\quad - i\sin(\Theta + kx_2 )\, \sigma_y^\alpha \sin k \theta+ i\cos(\Theta + kx_2 )\, \sigma_x^\alpha \sin k \theta \\
&=& \cos(\Theta+kx_2) X_{\alpha,k}^{(1)} + \sin(\Theta+kx_2) X_{\alpha,k}^{(2)}
\end{eqnarray*} 
The computation for $X_{\alpha,k}^{(2)}$ is identical.  This completes the proof.
\end{proof}



\section{An application of the hyperfunction fixed point localization formula to $\Omega SU(2)$}

In this section we will present our approach to computing a regularized Duistermaat-Heckman distribution on $\text{Lie}(T\times S^1)^*$ coming from the Hamiltonian action of $T\times S^1$ on $\Omega G$.  We will specialize to the case that $G = SU(2)$.  This problem (and the work herein) was originally motivated by Atiyah's approach to a similar problem \cite{atiyah1985circular}.  In that paper, Atiyah showed that the Atiyah-Singer index theorem is a consequence of applying the Duistermaat-Heckman localization formula to the loop space of a Riemannian manifold.  In \cite{atiyah1985circular}, Atiyah does also mention that similar methods can be applied to study $\Omega G$, however, no further details or specific theorems are provided.  Our original aim was to provide these details, as well as to study Duistermaat-Heckman distributions which come from Hamiltonian actions of compact tori on infinite dimensional manifolds.

It was discovered after completing this project that some of these issues had already been considered \cite{picken1989propagator}.  In this paper, Picken shows that the propagator for a quantum mechanical free particle moving on $G$ (with the invariant Riemannian metric coming from the Killing form) can be exactly expressed by applying the fixed point localization formula for $\Omega G$.  In this case, the ill defined left hand side of the localization formula for $\Omega G$ is expressed as a path integral on $G$, while the right hand of the localization formula tells us exactly how to express the result of this path integral in terms of solutions to the classical equations of motion.  We should highlight where our approach differs from his:  

\begin{enumerate}
	\item Throughout, Picken uses a variable $\varphi$ as a coordinate on $\mathfrak{t}$.  We will be calling this coordinate $x_1$ in our work.
	\item Picken is implicitly setting $x_{2} = 1$ throughout (i.e. he considers the slice $\mathfrak{t}\times\left\{1\right\} \subseteq \text{Lie}(T\times S^1)$.  This is evident in his choice of action functional, where the kinetic energy term:
	\[ I_{k}[g] = \int \langle g^{-1}\dot{g}, g^{-1}\dot{g}\rangle\,d\theta \]
	appears without a mass coefficient.  
	\item We will directly apply a fixed point localization formula to $\Omega G$ with its $T\times S^1$ action, and interpret the result as a \textit{hyperfunction} on $\text{Lie}(T\times S^1)$.  The advantage to this approach is that we will be able to Fourier transform this hyperfunction to obtain a closed form of a density function for what one should expect is the pushforward of the ``Liouville measure'' from $\Omega G$ to $\text{Lie}(T\times S^1)^*$ using the moment map.  Picken's formula is limited in this regard, since he does not use the localization formula to obtain a distribution on $\text{Lie}(T\times S^1)$ - he only obtains its restriction to a slice through $E = 1$.  He also makes no use of hyperfunctions in his paper.
\end{enumerate}

\begin{define}
Let $\gamma \in \Omega G^{T\times S^1}$.  The regularized equivariant Euler class of the normal bundle to $\gamma$ is defined to be the holomorphic function on $\text{Lie}(T\times S^1)_{\mathbb{C}}$ given by:
\[ e_{\gamma}^{T\times S^1}(z_1,z_2) = \prod_{k=1}^{\infty} \left( \prod_{\alpha \in \Delta} \frac{\lambda_\alpha^k(z_1,z_2)}{k z_2} \right) \]
\end{define}

The difference between the ``usual'' and the regularized equivariant Euler class of the normal bundle to $\gamma$ is that we divide out by $k z_2$ on each weight.  The regularization can be justified in a number of ways.  We will see shortly that when we include the regularizing terms, the resulting infinite product will converge to a useful functional expression for $e^{T\times S^1}_\gamma$.  Without the regularization, the infinite product does not converge.  Picken's work provides another justification for the regularization, since the resulting regularized localization formula provides an exact determination of the quantum mechanical propagator for a free particle moving on $G$.  

For simplicity, let's examine the example $G = SU(2)$.  We always use coordinates on $\text{Lie}(T\times S^1)$ consisting of the coroot basis for $\mathfrak{t}$, and normalize the $E$-component of the moment map so that:
\[ E\left( \begin{array}{cc} e^{i\theta} & 0 \\ 0 & e^{-i\theta} \end{array}\right) = 1/2 \]
Let $z = (z_1, z_2) \in \mathrm{Lie}(T\times S^{1})$ and let $\gamma(\theta) = \exp(i\eta\theta) \in \Omega SU(2)$ be a fixed point of the $T\times S^{1}$ action.  When we work with $G = SU(2)$ a choice of $\eta$ is really just a choice of integer, so for $\alpha \in \Delta$ the non-zero positive root we set $\alpha(\eta) = 2n$.  For every $k$ we get four weights for the isotropy representation, corresponding to the two root vectors in $\mathfrak{sl}_{2}$ and a the two weights cominig from a non-zero element of the Cartan subalgebra:
\begin{align*}
\lambda_{h,i}^{(k)}(z_1,z_2) &= kz_2\quad i =1,2 \\
\lambda_e^{(k)}(z_1,z_2)  &= kz_2 + 2(z_2 n+z_1) \\
\lambda_f^{(k)} (z_1,z_2) &= kz_2 - 2(z_2 n+ z_1) 
\end{align*}

\begin{prop}
Let $G = SU(2)$.  If $\gamma_n \in \Omega G^{T\times S^1}$, then the regularized equivariant Euler class of the normal bundle to $\gamma_n$ is given by:
\begin{equation}
e_{\gamma_n}^{T}(z_{1},z_{2}) = \frac{\sin\left(2\pi(n+z_1/z_2)\right)}{2\pi(n+z_1/z_2)}
\end{equation}
\end{prop}

\begin{proof}
Since the fixed points of the $T\times S^1$ action are isolated we have that the normal bundle to the fixed point set is simply $T_{\gamma_n}\Omega G$.  We can compute the regularized equivariant Euler class of $T_{\gamma_n}\Omega G$ by taking the product over the weights appearing in the isotropy representation of $T\times S^1$ on $T_{\gamma_n}\Omega G$, according to Theorem \ref{thm:weightsofisotropy}:

\begin{align*}
e_{\gamma_n}^{T}(z_{1},z_{2}) &= \prod_{k = 1}^{\infty} \prod_{\alpha \in \Delta} \lambda_{\alpha}^{(k)}/kz_2 \\
&= \prod_{k=1}^{\infty} \left[1 + \frac{2(z_2n+z_1)}{z_2 k}\right]\left[1 - \frac{2(z_2\eta+z_1)}{z_2 k}\right] \\
&= \prod_{k=1}^{\infty} \left[ 1 - \left(\frac{2(z_2 n + z_1)}{z_2 k}\right)^2\right]  \\
&= \frac{\sin\left(2\pi(n+z_1/z_2)\right)}{2\pi(n+z_1/z_2)} 
\end{align*}
where the last line follows from the infinite product formula for $\sin(z)$.
\end{proof}

\noindent\textit{Remark}: In the more general case of $G = SU(n)$, each choice of positive root will give a difference of squares, which then translates to an extra $\sin(z)/z$ term in the final result.  We would then take a product over all the positive roots.

\vspace{0.2cm}

In what follows we will write $e_{\gamma_n}^{T\times S^1}(z_1,z_2) = e_n(z_1,z_2)$ for notational simplicity.  A formal application of the fixed point localization formula to $\Omega G$ would then yield the following expression, valid for $(x_1,x_2) \in \text{Lie}(T\times S^1)$ such that $e_n(x_1,x_2) \neq 0$:
\begin{equation}
\label{eqn:oscillatory}
 \int_{\Omega G} e^{\omega + i\langle\mu(\gamma),x\rangle} = \sum_{n \in \mathbb{Z}} e^{i(nx_1 + \frac{n^2}{2}x_2)} \frac{2\pi(n+x_1/x_2)}{\sin\left(2\pi(n+x_1/x_2)\right)}
 \end{equation}
We have not addressed what types of objects that equation \ref{eqn:oscillatory} asserts an equality of.  In the setting of a compact symplectic manifold with a Hamiltonian action of a compact torus, one is free to understand this to be an equality of distributions, and even an equality of density functions on some open set.  But for the purposes of $\Omega G$, this perspective is insufficient.  For instance, the Duistermaat-Heckman ``distribution'' is supposed to be obtained by taking the Fourier transform of the right hand side of equation \ref{eqn:oscillatory}, however, we can see that the expression obtained from the localization formula is not even integrable since it has poles, and even if we ignore the poles coming from the denominator, the numerator grows linearly in the $\xi_1$ variable.  The terms appearing in the localization formula for $\Omega G$ also have unpleasant limiting behaviour as $x_2 \to 0$.  The right hand side of the localization formula should not be interpreted as a distribution (and consequently, neither should the left hand side).  

There are further hints in \cite{guillemin1988kostant} which suggest that the localization formula for $\Omega G$ should be an expression positing an equality of two hyperfunctions.  Suppose for a moment that we are considering a Hamiltonian action of a torus $T$ on a finite dimensional vector space with weights $\alpha_1,\dots,\alpha_n$.  To each weight we can associate a constant coefficient differential operator $D_{\alpha_i}$ on $\mathfrak{t}^*$.  The Duistermaat-Heckman distribution is a solution to the differential equation:
\[ D_{\alpha_1}\dots D_{\alpha_n} (DH(x)) = \delta(x) \]
When $V$ is infinite dimensional and we have infinitely many weights (such as is the case for the isotropy representation of $T\times S^1$ on the tangent space to a fixed loop in $\Omega G$), then we are forced to consider differential operators of infinite order.  Infinite order differential operators do not even act on distributions.  For example, any infinite order differential operator on $\mathbb{R}$ cannot act on the Dirac delta distribution because of the classical theorem which states that any distribution supported at the origin must be a finite sum of the Dirac delta distribution and its derivatives.  Hyperfunctions (and the related concept of a microfunction) are a sheaf on which infinite order differential operators do have a well defined action.  Furthermore, the entire classical theory of distributions is subsumed by the theory of hyperfunctions, so it makes more sense to study the Duistermaat-Heckman distribution as a hyperfunction, rather than as a distribution.


We now begin our construction of the Picken hyperfunction of $\Omega SU(2)$.  Fix a polarizing vector of the form $\xi = (\delta, \delta') \in \text{Lie}(T\times S^1)$, with $\delta' > 2\delta > 0$.  For the chosen polarization, we must determine the structure of the polarized weights of the isotropy representation at each fixed point.  Recall for $p \in \Omega G^{T\times S^1}$, we defined a cone $\gamma_p$ as the intersection of the positive half spaces coming from the polarized weights.  We now let $p_n$ denote the $n$'th fixed point of the $T\times S^1$ action on $\Omega SU(2)$.  

\begin{prop}
\begin{enumerate}
\item If $n > 0$, then the weights of the isotropy representation at the $n$'th fixed point satisfy the following inequalities:
\begin{align*}
& \lambda^{(k)}_\alpha(\xi) > 0, \quad \alpha = +2, k \geq 1  \\
&  \lambda^{(k)}_\alpha(\xi) > 0, \quad \alpha = -2, k > 2n \\  
&  \lambda^{(k)}_\alpha(\xi) < 0, \quad \alpha = -2, k \leq 2n \\
\end{align*}
\item If $n < 0$, then the weights of the isotropy representation at the $n$'th fixed point satisfy the following inequalities:
\begin{align*}
& \lambda^{(k)}_\alpha(\xi) > 0, \quad \alpha = -2, k \geq 1  \\
&  \lambda^{(k)}_\alpha(\xi) > 0, \quad \alpha = +2, k \geq 2n \\  
&  \lambda^{(k)}_\alpha(\xi) < 0, \quad \alpha = +2, k  < 2n \\
\end{align*}
\item If $n = 0$, then the weights of the isotropy representation at $p_0$ satisfy the following inequalities:
\[ \lambda^{(k)}_\alpha (\xi) > 0, \quad \text{for all $\alpha = \pm 2$, $k \geq 1$} \]
\end{enumerate}
Consequently,
\begin{align*}
\gamma_{p_0} &= \left\{ (y_1,y_2) \in i\text{Lie}(T\times S^1)\,\vert\, \vert y_1 \vert < y_2/2 \right\} \\
\gamma_{p_n} &= \left\{ (y_1,y_2) \in i\text{Lie}(T\times S^1)\,\vert\, \vert y_1 \vert < y_2/2, y_1 > 0 \right\}\quad n\neq 0
\end{align*}
\end{prop}
\noindent\textit{Remark}: Since the cones $\gamma_{p_n}$ are independent of $n$ (so long as $n \neq 0$), after the proof of this proposition will will simply denote $\gamma_{\neq 0} := \gamma_{p_n}$ and $\gamma_0 := \gamma_{p_0}$
\begin{proof}
We shall prove the result for 1, as 2 and 3 are similar.  For the root $\alpha = +2$, we have that
\[ \lambda^{(k)}_\alpha(\xi) = k\delta' + 2(n\delta' + \delta) \]
This is a positive number, being a sum of positive numbers.  For the root $\alpha = -2$, we are interested in finding the values $k \geq 1$ such that:
\[ k\delta' - 2(n\delta' + \delta) < 0 \]
Dividing both sides by the positive number $\delta'$ yields
\[k - 2n - \frac{2\delta}{\delta'} < 0 \]
By our choice of polarization we have $0 < 2\delta/\delta' < 1$, so the above inequality is true exactly when $1 \leq k \leq 2n$, which proves the first claim.  

For a root $\alpha = \pm 2$ and $k \geq 1$, we denote $H_{k,\pm}^{(n)} = \displaystyle{ \left\{(y_1,y_2) \in i \text{Lie}(T\times S^1)\,\vert\, k y_2 \pm 2( ny_2 + y_1) > 0\right\} }$, which is the positive half plane corresponding to the weight $\lambda^{(k)}_{\alpha}$ at the $n$'th fixed point. 

We now consider the second set of claims about the cones $\gamma_{p_n}$ and $\gamma_{p_0}$.  First, consider the case where $n = 0$.  By part (3) of the previous work towards this proposition, we can see that the weights of the isotropy representation at the $n=0$ fixed point are already polarized.  Letting $\eta_k = H_{k,+} \cap H_{k,-}$, we have by definition that $\gamma_{p_0} = \bigcap_{k \geq 1} \eta_k$.  Notice that $k \geq k'$ implies that $\eta_k \supseteq \eta_{k'}$, and so $\gamma_{p_0} = \eta_1$.  But now the proof is complete, since
\[ \gamma_{p_0} = \eta_1 = \left\{ (y_1,y_2)\,\vert\, y_2 + 2y_1 > 0\right\} \cap \left\{ (y_1,y_2)\,\vert\, y_2 - 2y_1 > 0\right\} = \left\{ (y_1,y_2)\,\vert\, y_2 > 2 \vert y_1 \vert \right\} \]

The case where $n \neq 0$ is similar; the only modification required is that the set of polarized weights of the isotropy representation at the $n$'th fixed point is equal to the set of weights at the $n=0$ fixed point, with one extra weight of the form $(y_1,y_2) \mapsto 2y_1$.    
\end{proof}

Another consequence of the previous proposition is that
\[ (-1)^{p_n} = \left\{ \begin{array}{cc} 1 & n = 0 \\ 1 & n > 0 \\  -1 & n < 0 \end{array}\right. \]

We now have all the necessary pieces to construct the Picken hyperfunction for the $T\times S^1$ action on $\Omega SU(2)$, which we expect is a hyperfunction replacement for the sum over the fixed points appearing in the Duistermaat-Heckman localization formula.  

As before, we let $\tilde{\lambda}_{k,\alpha}$ denote the polarized weights of the isotropy representation at the $n$'th fixed point; we leave the $n$ implicit to avoid notational clutter.  For every $n$, we apply Lemma \ref{lemma:infiniteproduct} to the set of hyperfunctions $\left\{f^{(n)}_{\tilde{\lambda}_{k,\alpha}}(x)\right\}_{k=1}^{\infty}$ (c.f. notation of Corollary \ref{corollary:hypereuler}, making sure to use the regularized weights to guarantee uniform convergence of the infinite product.  The resulting hyperfunction is the regularized equivariant Euler class to the normal bundle of the $n$'th fixed point:

\[ \frac{1}{e_n(x_1,x_2)} = b_{\gamma_{p_n}} \left( \frac{2\pi(n+z_1/z_2)}{\sin(2\pi z_1/z_2)} \right) \]

Putting all of these results together we obtain the Picken hyperfunction for the Hamiltonian $T\times S^1$ action on $\Omega SU(2)$:
\begin{align*}
L_{\Omega SU(2)}(x_1,x_2) &= \frac{1}{(2\pi i)^2} b_{\gamma_{\neq 0} } \left(  \sum_{ n > 0 } e^{iz_1 n + iz_2 n^2/2}\frac{2\pi(n+z_1/z_2)}{\sin(2\pi z_1/z_2)} - \sum_{ n < 0 } e^{iz_1 n + iz_2 n^2/2}\frac{2\pi(n+z_1/z_2)}{\sin(2\pi z_1/z_2)} \right)  \\
&\qquad + \frac{1}{(2\pi i)^2} b_{\gamma_0} \left( \frac{2\pi z_1/z_2 }{\sin(2\pi z_1/z_2)} \right)
\end{align*}
Ultimately, we would like to be able to take a Fourier transform of the Picken hyperfunction in order to obtain the Duistermaat-Heckman hyperfunction.  The following proposition guarantees that the Picken hyperfunction of $\Omega SU(2)$ is in the class of hyperfunctions which have Fourier transforms, and so guarantees that we can find some hyperfunction analogue of the Duistermaat-Heckman distribution in this infinite dimensional example.  We will do this term by term.

\begin{prop}
\label{prop:slowlyincreasing}
For every $n$,
\[ I_n(z_1,z_2) = \frac{n + z_1/z_2}{\sin(2\pi z_1/z_2)} \]
is a slowly increasing holomorphic function on $\mathbb{R}^2\times i\gamma_{p_n}\subseteq \text{Lie}(T\times S^1)_{\mathbb{C}}$.  
\end{prop}
\begin{proof}
That the function in question is holomorphic on $\mathbb{R}^2 \times i\gamma_{p_n}$ follows from its expression as an infinite product of regularized weights, and that the cones $\gamma_{p_n}$ are constructed to avoid the zero locus of all such weights.  It remains to show that $I_n(z_1,z_2)$ is slowly increasing.

For fixed $(y_1,y_2) \in \gamma_n$, the image of the curves $x_1 = mx_2$ ($m \in \mathbb{R}$) under the mapping $(z_1,z_2) \mapsto z_1/z_2$ are the parametric curves given by:
\[ \mathbb{R} \to \mathbb{C} \]
\[ s \mapsto \frac{ms^2 + y_1y_2}{s^2+y_2^2} + i \frac{sy_1 -msy_2}{s^2 + y_2^2} \]
These are easily seen to be ellipsoidal arcs which cross the real axis at $\text{Re}(z_1/z_2) = y_1/y_2$ when $s=0$, and asymptotically approach the real axis from above (below) at $\text{Re}(z_1/z_2) = m$  as $s \to \infty$ when $m > 0$ (and from below the axis if $m < 0$).  

We assume $n \neq 0$, since the $n = 0$ case is similar.  Fix a compact set $K\subseteq \gamma_{p_n}$ and any $\epsilon > 0$.  Since $(y_1,y_2) \mapsto y_1/y_2$ is continuous on $K$ it will achieve its maximum and minimum, so there is a $\delta >0$ such that the estimate $\delta \leq y_1 / y_2 \leq 1/2 - \delta$ holds uniformly over $K$.  Since the numerator of $I_n(z_1,z_2)$ is slowly increasing (it is a polynomial), it suffices to prove that:
\[ \bigg| \frac{e^{-\epsilon\,\vert \text{Re}(z) \vert}}{\sin(2\pi z_1/z_2)} \bigg| \to 0 \]
uniformly in $K$ as $\text{Re}(z) \to \infty$.  

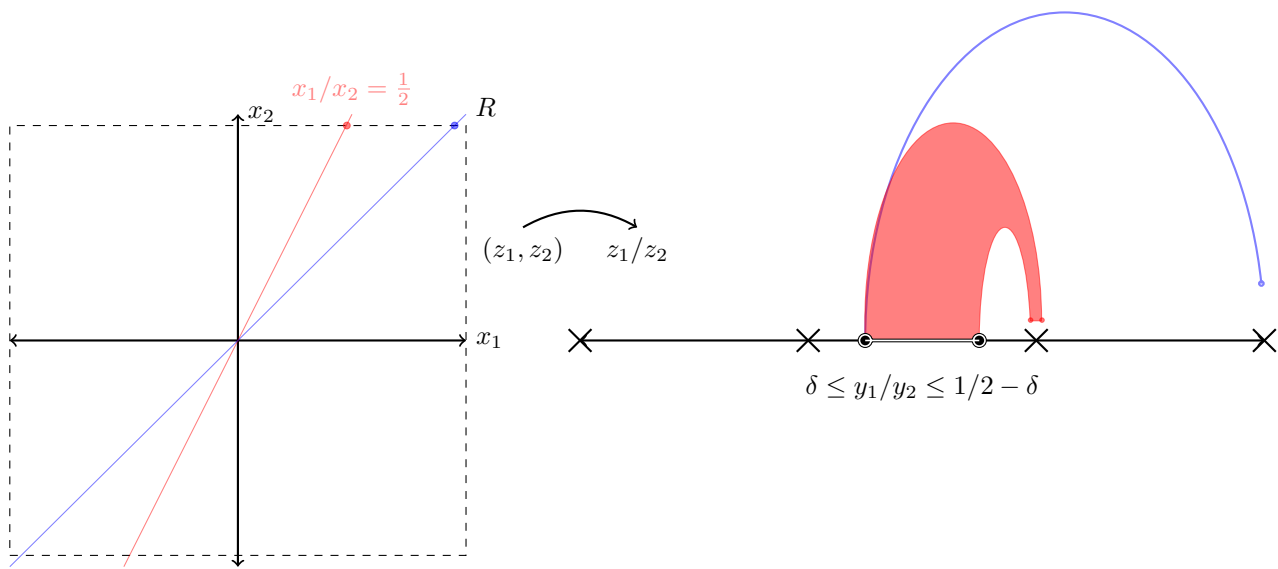
\begin{figure}[h]
\begin{center}
\begin{tikzpicture}[scale = 1.5]
				
				\draw[black,thick,->] (-0.5,1) node[anchor=north]{$(z_1,z_2)$} to[bend left] (0.5,1) node[anchor=north]{$z_1/z_2$};
				
				\draw[black,thick,<->] (-5,0) -- (-1,0) node[anchor=west]{$x_1$};
				\draw[black,thick,<->] (-3,-2) -- (-3,2) node[anchor = west]{$x_2$};
				\draw[dashed,black] (-5,-1.9) rectangle (-1,1.9) node[anchor = south west]{$R$};
				
				\draw[blue,opacity=0.5,fill] (-1.1,1.9) circle[radius=0.03];
				
				\draw[red,opacity=0.5,fill] (-2.045,1.9) circle[radius=0.03];
				
				\draw[red,opacity=0.5] (-4,-2) -- (-2,2) node[anchor = south]{$x_1/x_2 = \frac12$};
				\draw[blue,opacity=0.5] (-5,-2) -- (-1,2) ;

				\draw[black,thick] (3.9,0.1) -- (4.1,-0.1);
				\draw[black,thick] (3.9,-0.1) -- (4.1,0.1);
				
				\draw[black,thick] (5.9,0.1) -- (6.1,-0.1);
				\draw[black,thick] (5.9,-0.1) -- (6.1,0.1);
				
				\draw[black,thick] (-0.1,0.1) -- (0.1,-0.1);
				\draw[black,thick] (-0.1,-0.1) -- (0.1,0.1);
				
				\draw[black,thick] (1.9,0.1) -- (2.1,-0.1);
				\draw[black,thick] (1.9,-0.1) -- (2.1,0.1);
				
				\draw[black,thick,<->] (0,0) -- (6,0) node[anchor=west]{};
				
				\draw[blue,thick,opacity=0.5] (2.5,0) arc(180:10: 1.75 and 2.9) -- ++(0,0) circle[radius=0.02,fill];
				
				 \draw[red,fill,opacity=0.5] (3.5,0) arc (180:10: 0.225 and 1) -- ++(0.1,0) -- ++(0,0) arc (0:185.5: 0.775 and 1.75) -- cycle; 
				 \draw[red,fill,opacity=0.5] (3.95,0.18) circle[radius=0.02];
				 \draw[red,fill,opacity=0.5] (4.05,0.18) circle[radius=0.02];
				 
				 \draw[black,double,fill] (3.5,0) circle[radius =0.05] -- (3,0) node [label={[label distance=0.2cm]below:{$\delta \leq y_1/ y_2 \leq 1/2 - \delta$}}]{} -- (2.5,0) circle[radius=0.05];
\end{tikzpicture}
\caption{Proof that $I_n(z_1,z_2)$ is slowly increasing.  The left side of the figure demonstrates the $(x_1,x_2)$ plane; the right hand side is demonstrating the image of the map $(z_1,z_2) \mapsto z_1/z_2$ when we fix various values of $(y_1,y_2)$.  The red filled region is showing the image of the line $x_1 = x_2/2$ as $(y_1,y_2)$ varies over $K$, with $\max\left\{\vert x_1\vert,\vert x_2\vert\right\} \leq R$.  The blue curve is showing the image of the line $x_1 = x_2$ (fixing $(y_1,y_2)$ such that $y_1/y_2 = \delta$).  The poles of $\csc(2\pi z)$ are demonstrated with $\times$.}
\label{fig:slowlyincreasing}
\end{center}
\end{figure}

First, we notice that if we fix $y_1/y_2$ as above, then for every $R$ sufficiently large we have:
\[ \max\left\{\vert x_1\vert, \vert x_2\vert\right\} = R \Rightarrow \vert \csc(2\pi z_1/z_2) \vert \leq \bigg| \csc\left(2\pi \frac{R^2/2 + y_1y_2}{R^2+y_2^2} + i \frac{Ry_1 -Ry_2/2}{R^2 + y_2^2}\right) \bigg| \]
This estimate follows from the observation that the maximum of $\csc(2\pi z_1/z_2)$ on the box occurs at the point $(x_1,x_2)$ such that the distance from $z_1/z_2$ to a pole of $\csc(2\pi z)$ is minimized; this condition is satisfied on the line $x_1 = x_2/2$.  A uniform bound over $K$ can be found because of our previous estimate on $y_1/y_2$.  Figure \ref{fig:slowlyincreasing} demonstrates these estimates.  The proof is completed by noticing that $\csc(2\pi z_1/z_2)$ has linear growth (which is dominated by any exponential) as $x_2 \to \infty$ because all of its poles are simple. 
\end{proof}

By Proposition \ref{prop:slowlyincreasing}, $L_{\Omega SU(2)}(x_1,x_2)$ is a slowly increasing hyperfunction, so we may take its Fourier transform.  Let $S_n$ be a contour in $\text{Lie}(T\times S^1)_{\mathbb{C}}$ chosen so that $(y_1,y_2) \in \gamma_{p_n}$.   After choosing a holomorphic partition of unity $\chi_\sigma(z)$, we may write the following expression for the Duistermaat-Heckman hyperfunction:

\[ DH(\xi_1,\xi_2) = \frac{1}{(2\pi i)^2} \sum_{\sigma \in \Sigma}  \sum_{n \in \mathbb{Z}} b_{-\sigma^\circ}\left( \int_{S_n} e^{-i(\zeta_1- n)z_1 -i(\zeta_2- n^2/2)z_2}\frac{2\pi(n+z_1/z_2)}{\sin(2\pi z_1/z_2)} \chi_\sigma(z_1,z_2) \,dz_1\,dz_2 \right)\]

One might try and proceed with the computation of this integral, as in the example of section \ref{HyperHam}; however, if one uses the standard holomorphic partition of unity then the computation of the contour integrals by a method of iterated residues becomes very complicated.  The difficulty essentially arises from the fact that the integrand of the resulting multivariable contour integral has a polar locus consisting of triples of lines that intersect.  If one uses the following partition of unity:
\[ 1 = \frac{1}{1+e^{z_1}}\frac{1}{1+e^{\pi z_2}} + \frac{1}{1+e^{-z_1}}\frac{1}{1+e^{\pi z_2}} +\frac{1}{1+e^{z_1}}\frac{1}{1+e^{-\pi z_2}} +\frac{1}{1+e^{-z_1}}\frac{1}{1+e^{-\pi z_2}} \]
then polar locus of the integrand defining the Fourier transform consists of isolated singularities which are locally cut out by a pair of equations.  The residues near such singularities are readily computed, but do not appear to re-sum in any obvious way.  We leave a further examination of the form of the Duistermaat-Heckman hyperfunction of $\Omega SU(2)$ as an open problem.

\bibliographystyle{amsalpha}
\bibliography{hyperfnbib}

\end{document}